\documentclass[12pt, reqno]{amsart} \usepackage[margin=2.3cm]{geometry}\openup .6em 

\usepackage{amsmath,amssymb,amsfonts,amstext,verbatim,amsthm,mathrsfs}
\usepackage[mathcal]{eucal}
\usepackage[all]{xy}

\theoremstyle{plain}
\newtheorem{thm}{Theorem}[section]

\newtheorem{cor}[thm]{Corollary}
\newtheorem{lemma}[thm]{Lemma}

\newtheorem{defn}[thm]{Definition}
\newtheorem{definition}[thm]{Definition}
\theoremstyle{remark}
\newtheorem{remark}[thm]{Remark}

\makeatletter
\def\makeCal#1{\expandafter\newcommand\csname c#1\endcsname{\mathcal{#1}}}
\def\makeBB#1{\expandafter\newcommand\csname b#1\endcsname{\mathbb{#1}}}
\def\makeFrak#1{\expandafter\newcommand\csname f#1\endcsname{\mathfrak{#1}}}
\count@=0 \loop \advance\count@ 1 \edef\y{\@Alph\count@}
\expandafter\makeCal\y \expandafter\makeBB\y \expandafter\makeFrak\y
\ifnum\count@<26 \repeat

\newcommand{\todo}[1]{}

\newcommand {\id}{\operatorname{id}}
\newcommand{\im}{\operatorname{im}}
\newcommand{\coker}{\operatorname{coker}}

\newcommand {\Hom}{\operatorname{Hom}}

\newcommand{\End}{\operatorname{End}}
\newcommand{\lEnd}{\operatorname{\mathcal{E}nd}}

\newcommand {\<}{\langle}
\renewcommand {\>}{\rangle}
\newcommand{\tensor}{\otimes}
\renewcommand{\O}{\mathscr{O}}
\newcommand{\isom}{\cong}
\newcommand{\mat}[4]{\begin{pmatrix}#1&#2\\#3&#4\end{pmatrix}}
\newcommand{\vect}[2]{\begin{pmatrix}#1\\#2\end{pmatrix}}

\newcommand{\hk}{hyperk{\"a}hler }

\newcommand{\SL}{\operatorname{SL}}
\newcommand{\GL}{\operatorname{GL}}
\newcommand{\can}{\operatorname{can}}
\newcommand{\lra}{\longrightarrow}
\newcommand{\lRa}[1]{\stackrel{#1}{\longrightarrow}}

\newcommand{\GM}{\rm GM}

\newcommand{\hash}{\#}

\renewcommand{\leq}{\leqslant}

\newcommand{\tr}{\operatorname{tr}}
\newcommand{\br}{\operatorname{br}}
\newcommand{\Quad}{\operatorname{Quad}}

\newcommand{\Bun}{\operatorname{Bun}}
\newcommand{\Flat}{\operatorname{Flat}}
\newcommand{\Higgs}{\operatorname{Higgs}}
\newcommand{\Spec}{\operatorname{Spec}}
\newcommand{\Ext}{\operatorname{Ext}}
\newcommand{\Sch}{\operatorname{Sch}}
\newcommand{\Sets}{\operatorname{Sets}}
\newcommand{\PT}{\operatorname{PT}}
\newcommand{\Coh}{\operatorname{Coh}}

\begin{document}

\title{Joyce structures on spaces of quadratic differentials}
\begin{abstract}Consider the space parameterising complex projective curves of genus $g$ equipped with a quadratic differential  with simple zeroes. We use the geometry of isomonodromic deformations to construct a complex \hk structure on the total space of its tangent bundle. This provides non-trivial examples of the Joyce structures  introduced in \cite{RHDT2} in relation to Donaldson-Thomas theory. 
 \end{abstract}

\author{Tom Bridgeland}

\maketitle


\section{Introduction}


Since Hitchin's classic papers  \cite{Hitchin,Hitchin2} moduli spaces of Higgs bundles on algebraic curves have  appeared in many areas of pure mathematics and mathematical physics. Consider for definiteness the space  $\cM_C(E,\Phi)$ parameterising $\SL_2(\bC)$ Higgs bundles  $(E,\Phi)$ on a smooth complex projective curve  $C$ of genus $g>1$. Two of its most important features are (i) a \hk metric,  defined using the non-Abelian Hodge correspondence, 
and (ii) a proper map $\cM_C(E,\Phi)\to H^0(C,\omega_C^{\tensor 2})$, whose target parameterises quadratic differentials on $C$, and 
 whose general fibres  are abelian varieties, homeomorphic to $(S^1)^{6g-6}$.
 
In  this paper we study a  moduli space  $\cM(C,E,\nabla,\Phi)$ which in some respects resembles  a complexification of $\cM_C(E,\Phi)$. For a given integer  $g>1$ it  parameterises the data of a curve $C$ of genus $g$, together with an $\SL_2(\bC)$  bundle $E$  on $C$, equipped with both a flat connection $\nabla$ and a Higgs field $\Phi$.   We construct (i)   a meromorphic complex \hk metric, defined using isomonodromic flows, and (ii) a map  $\cM(C,E,\nabla,\Phi)\to \cM(C,Q)$ whose fibres are algebraic tori  $(\bC^*)^{6g-6}$. The target $\cM(C,Q)$ of this map is the `generic Hitchin base'   parameterising curves $C$ of genus $g$ equipped with a quadratic differential $Q\in H^0(C,\omega_C^{\tensor 2})$ with simple zeroes. 

It is important to note  that the complex \hk metric we construct on $\cM(C,E,\nabla,\Phi)$ is a much simpler object than the \hk metric on  $\cM_C(E,\Phi)$. In particular it is algebraic, 
in contrast to the Hitchin metric which  is  highly transcendental. A striking demonstration of this difference appears on generalising to the setting  where the Higgs fields and connections have poles of fixed orders \cite{Z}. One can then take the curve $C$ to have genus $g=0$, and in simple examples the resulting complex \hk structures can be written explicitly in terms of rational functions \cite{A2}. No such explicit formulae are expected for the Hitchin metric.

We expect our construction  to generalise to gauge groups other than $\SL_2(\bC)$ but this will require  new ideas. A key point in our construction is that the space of quadratic differentials  $H^0(C,\omega_C^{\tensor 2})$ is the cotangent fibre to the moduli space of curves. To generalise to the gauge group $G=\SL_m(\bC)$, for example, would require  a  space of `higher complex structures' whose cotangent fibre is the Hitchin base $\bigoplus_{k=2}^m H^0(C,\omega^{\tensor k})$. A natural candidate  is provided by the  work of Fock and Thomas \cite{FT}. The expectation is then that the total space of the cotangent bundle of this space  should carry a meromorphic Joyce structure generalising the one constructed here.  

\subsection{DT invariants and [GMN]}

In a celebrated development \cite{GMN1,GMN2} Gaiotto, Moore and Neitzke   uncovered a deep relation between the \hk geometry of the moduli spaces $\cM_C(E,\Phi)$ and the BPS invariants of a  class of four-dimensional  $N=2$ supersymmetric gauge theories known as theories of class $S[A_1]$. More precisely, they introduced a class of  non-linear Riemann-Hilbert (RH) problems defined by the BPS invariants, and showed that their solutions describe twistor lines in the twistor space of $\cM_C(E,\Phi)$. In mathematical terms these BPS invariants can be understood as the Donaldson-Thomas (DT) invariants of a certain three-dimensional Calabi-Yau (CY$_3$) triangulated category \cite{BS}. 

Our interest in the  moduli spaces  $\cM(C,E,\nabla,\Phi)$ stems from a general programme  \cite{RHDT1,RHDT2} which attempts to encode the  DT invariants of a CY$_3$  triangulated category  in a geometric structure on its space of stability conditions. 
This procedure is currently  highly conjectural, and involves a class  of RH problems closely related to those considered by Gaiotto, Moore and Neitzke, and obtained from them by a procedure  known in physics as the conformal limit \cite{G}.  In this limit it appears that the geometry of the Hitchin space considered in \cite{GMN2} should be replaced by the simpler geometry of the space $\cM(C,E,\nabla,\Phi)$ considered here.

The geometric structure  on spaces of stability conditions envisaged in \cite{RHDT2}  is a kind of non-linear Frobenius structure, and was christened a Joyce structure  in honour of the paper \cite{HolGen} where the main ingredients were first identified. In later work with Strachan \cite{Strachan} it was shown that a Joyce structure on a complex manifold   induces a complex \hk structure on  the total space  of its tangent bundle. The main result of this paper is a construction of  a meromorphic Joyce structure on the space $\cM(C,Q)$. 

The space $\cM(C,Q)$ can be identified with a space of stability conditions on a CY$_3$ category by the work of Haiden \cite{Haiden}. We leave for future research the problem of using the Joyce structure constructed here to solve the RH problems of \cite{RHDT1} defined by  the DT invariants of Haiden's category. This would probably be more easily accomplished in the setting of meromorphic quadratic differentials \cite{BS,Z}, using the Fock-Goncharov cluster structure  on the wild character variety \cite{FG} and results from exact WKB analysis as in \cite{GMN2}. One particular example was treated in detail in this way in \cite{A2}, and other partial results have been obtained  \cite{All0,All2}. 

\subsection{Summary of the construction}
\label{th}


We fix a genus $g>1$ throughout the paper. The space $\cM(C,Q)$ as defined above is a smooth Deligne-Mumford stack or, if we work in the analytic category, a complex orbifold. Since this may be uncomfortable for some readers, we shall also fix  an  integer $\ell>0$, and insist that all curves $C$ are equipped with a level $\ell$ structure. This  extra data plays no essential role in our constructions so we will omit it from the notation. We always assume $\ell>2$ since this has the pleasant consequence that  all moduli spaces appearing are smooth quasi-projective varieties. But the  reader happy with stacks  can eliminate the  level structures by taking $\ell=1$ and working instead with smooth Deligne-Mumford stacks.

Let us then introduce the smooth quasi-projective variety $M=\cM(C,Q)$ parameterising pairs $(C,Q)$  consisting of a smooth projective complex curve $C$, equipped as always with a level $\ell$ structure, and a quadratic differential $Q\in H^0(C,\omega_C^{\tensor 2})$ with simple zeroes.  Associated to a point $(C,Q)\in M$ is a smooth spectral curve $\Sigma$ cut out in the cotangent bundle $T^*_C$ by the equation $y^2=Q$. The projection $p\colon \Sigma\to C$  is a branched double cover  with a covering involution $\sigma$. The tangent space to $M$ at the point $(C,Q)$ can then be identified with the anti-invariant cohomology group $H^1(\Sigma,\bC)^-$. The dual of the integral anti-invariant  homology defines an integral affine structure $T_M^{\bZ}\subset T_M$ and we consider  the  quotient   $X^\hash=T_M/T_M^{\bZ}$. The fibre of the induced projection $\pi\colon X^\hash\to M$   over the point $(C,Q)$ is the quotient of the group $H^1(\Sigma,\bC^*)^-$ by the finite subgroup $p^*(H^1(C,\{\pm 1\}))$ and is isomorphic to $(\bC^*)^{6g-6}$

A Joyce structure on $M$ is essentially the data of a pencil of flat symplectic non-linear connections $h_\epsilon$ on the bundle $\pi\colon X^\hash\to M$ parameterised by $\epsilon\in \bC^*$. The associated complex \hk structure on $X^\hash$ is then defined by taking the eigenspaces of the operators $I,J,K$ to be the horizontal sub-bundles of $T_{X^\hash}$ defined by certain elements of this pencil.
We construct the non-linear connections $h_\epsilon$ as follows. 
We can realise elements of  $H^1(\Sigma,\bC^*)^-$ as the holonomy of anti-invariant line bundles with connection $(L,\partial)$ on $\Sigma$. The usual spectral correspondence associates to $L$ a Higgs bundle $(E,\Phi)$ on the curve $C$. Using an extension of this  correspondence to connections, valid under a genericity assumption on $L$, we can  use $\partial$ to induce a connection $\nabla$ on $E$. The required family of non-linear connections $h_\epsilon$ is then given by the isomonodromy flows for the connections $\nabla-\epsilon^{-1}\Phi$.

The complex \hk structure we construct on $M$  has poles; these arise from two interesting issues.
Firstly, the extension of the spectral correspondence to connections requires a genericity assumption on the line bundle $L$. This relates to  the theta divisor in the generalised Prym variety of the double cover $p\colon \Sigma\to C$. Secondly, given a fixed curve $C$ equipped with a Higgs bundle $(E,\Phi)$, we need to lift deformations of the quadratic differential $Q=\tr(\Phi^2)$ to deformations of the Higgs field $\Phi$. This relates to the wobbly locus in the space of Higgs bundles \cite{DP}. 

The extended spectral correspondence  in our construction can be 
viewed as an abelianization procedure for flat connections in the presence of a quadratic differential.  In the case of meromorphic quadratic differentials, this de Rham abelianization  can be compared with the Betti abelianization  of \cite{HN,Nik1} which depends on the choice of a spectral network on $C$. Their relationship  is highly non-trivial, and in fact, if we take the spectral network to be the WKB triangulation of the quadratic differential, one can view the solutions to the RH problems discussed above as intertwining these two abelianisation procedures.

\subsection*{Plan of the paper} The aim of the paper is to construct a meromorphic Joyce structure on the space $M=\cM(C,Q)$. A Joyce structure on a complex manifold $M$  is a combination of two ingredients: a period structure and a pencil of non-linear connections on the tangent bundle. The  definitions of all these terms can be found in Section \ref{general}.

The required period structure on $M$ is well-known and is described in Section \ref{three}. In Section \ref{four} we recall the standard correspondence between Higgs bundles on $C$ and line bundles on the spectral curve $\Sigma$, and explain how it can be  extended to bundles with connection. Section \ref{five}  introduces the essential diagrams of moduli spaces which will be used to construct the pencil of non-linear connections. We also prove two crucial generic finiteness results.

In Section \ref{six}  we recall the Atiyah-Bott symplectic form on the moduli space of flat connections and prove that our extended spectral correspondence preserves it. The meromorphic Joyce structure on $M$ is finally constructed in Section \ref{seven} using  isomonodromic flows. In Section \ref{slices}, we describe an interesting compatibility relation  between this Joyce structure  and the Lagrangian submanifolds in $M$ obtained by fixing the curve $C$.

We include in Appendix \ref{moduli}  a summary of the scheme-theoretic definitions and constructions of the various moduli spaces used in the main text.

\subsection*{Conventions and notation} 
We use rather unconventional conventions for labelling moduli spaces. In general a symbol $\cM_A(B)$ denotes the moduli space of objects of type $B$ on a fixed object $A$. So for example $\cM_C(E,\Phi)$ denotes the moduli space of $\SL_2(\bC)$ Higgs bundles $(E,\Phi)$ on a fixed curve $C$, whereas $\cM(C,E,\Phi)$ denotes the moduli space where $C$ is also allowed to vary. We can only apologise for the initially nonsensical appearance of statements such as `Take a point $(C,Q)\in \cM(C,Q)$', and hope that this proves less of an inconvenience  than having to constantly consult a dictionary of  the large number of moduli spaces that appear.  

The paper contains many connections, both linear and non-linear. Linear connections on a vector bundle $E$ are specified by their covariant derivative $E\to E\tensor \Omega^1$ and are usually denoted by the symbols $\nabla$ or $\partial$. Non-linear connections on a map $\pi\colon X\to M$ are specified by a bundle map $\pi^*(T_M)\to T_X$ and are denoted by small latin letters $h$, $j$ etc.  
Throughout the paper we encounter families of connections parameterised by $\epsilon\in \bC^*$, which we refer to as pencils. Often the inverse $\zeta=\epsilon^{-1}$ would seem to be a more natural parameter, but we will nonetheless use $\epsilon$ since in the relations with mathematical physics this is the most natural variable, relating variously to the string coupling, Planck's constant, etc.

We work with both complex manifolds and algebraic varieties. Except in the Appendix, all algebraic varieties appearing are smooth and quasi-projective over $\bC$. We view them as a subcategory of the category of complex manifolds. The derivative of a map of complex manifolds $f\colon X\to Y$ is denoted $f_*\colon T_X\to f^*(T_Y)$. The map $f$ is called {\'e}tale if $f_*$ is an isomorphism.

\subsection*{Acknowledgements} I am very grateful for   correspondence and conversations with  Dylan Allegretti, Anna Barbieri,  Maciej Dunajski,   Dima Korotkin,  Andy Neitzke, Nikita Nikolaev,  Ivan Smith, Ian Strachan, J{\"o}rg Teschner and Menelaos Zikidis.  Extended discussions with Zikidis were particularly important in getting the results into the final form presented here. I am supported by a Royal Society Research Professorship at the   University of Sheffield.


\section{Joyce structures}
\label{general}

This section introduces the notion of a Joyce structure on a complex manifold $M$. The definition arose from a line of work in Donaldson-Thomas theory \cite{RHDT2} which originated with a paper of Joyce \cite{HolGen}. We first define the notion of a pre-Joyce structure, which consists of a pencil of flat symplectic non-linear connections on the tangent bundle $\pi\colon X=T_M\to M$. Following \cite{Strachan}, we show that a pre-Joyce structure on $M$ induces a complex \hk structure on $X$. This  construction  is  well known in the twistor-theory literature, see for example \cite{DM}, and goes back to the work of Pleba{\'n}ski \cite{P}. A Joyce structure is then defined to be a  pre-Joyce structure with certain extra symmetries. The description of these symmetries involves a  strengthening of the notion of an integral affine structure which we call a period structure.

\subsection{Non-linear connections}
\label{connect}

We begin by briefly summarising some basic facts about non-linear  connections in the sense of Ehresmann.\todo{Give reference for Ehresmann connections} We work with complex manifolds and holomorphic maps, but everything in this section holds also in the smooth setting. 
 
Let $\pi\colon X\to M$ be a holomorphic submersion of complex manifolds. We denote the fibres by $X_m=\pi^{-1}(m)$. The derivative of $\pi$ gives rise to a  short exact sequence of vector bundles
\begin{equation}\label{bass}0\lra T_{X/M}\lRa{i} T_X\lRa{\pi_*} \pi^*(T_M)\lra 0.\end{equation}
\begin{defn}
A non-linear
connection on  the map $\pi$ is a    bundle map $h\colon \pi^*(T_M)\to T_X$ satisfying $\pi_*\circ h=1$.\end{defn}
Writing $H=\im(h)$ and $V=T_{X/M}$, the tangent bundle decomposes as a direct sum $T_X=H\oplus V$. We call tangent vectors and vector fields horizontal or vertical if they lie in $H$ or $V$ respectively. 
Note that a vector field $u\in H^0(M,T_M)$ can be lifted to a horizontal vector field $h(u)\in H^0(X,T_X)$ by composing  the pullback $\pi^*(u)\in H^0(X,\pi^*(T_M))$ with the map $h$.

Consider a smooth path $\gamma\colon [0,1]\to M$. Given a point  $x\in X_{\gamma(0)}$ we can look for a lifted path  $\alpha\colon [0,\delta]\to X$ satisfying $\alpha_*(\frac{d}{dt})=h(\gamma_*(\frac{d}{dt}))$ and $\alpha(0)=x$. Such a  lift will exist for small enough $\delta>0$. For $t\in [0,\delta]$ we call $\alpha(t)\in X_{\gamma(t)}$ the time $t$ parallel transport of the point $x$ along the path $\gamma$. Given a point $x_0\in X_{\gamma(0)}$ we can find a $\delta>0$ and open subsets  $U_t\subset X_{\gamma(t)}$ with $x_0\in U_0$, such that time $t$ parallel transport along $\gamma$ defines an  isomorphism $\PT_{\gamma}(t)\colon U_0\to U_t$ for each $t\in [0,\delta]$.

Given  complex manifolds $M,N$ there is a connection on the projection map $\pi_M\colon M\times N\to M$ induced by the canonical splitting $T_{M\times N}=\pi_M^*(T_M)\oplus \pi_N^*(T_N)$.  A connection $h$ on  $\pi\colon X\to M$ is called flat if it is locally isomorphic to a connection of this form. More precisely:

\begin{definition}
The connection $h$ is flat if the following equivalent conditions hold:
\begin{itemize}
\item[(i)] for every $x\in X$ there are local co-ordinates  $(x_1,\cdots, x_n)$ on $X$ at $x$, and $(y_1,\cdots, y_d)$ on $M$ at  $\pi(x)$, such that $x_i=\pi^*(y_i)$ and  $h(\frac{\partial}{\partial y_i})=\frac{\partial}{\partial x_i}$ for $1\leq i\leq d$,
\item[(ii)]  
the sub-bundle $H=\im(h)\subset T_X$ is involutive: $[H,H]\subset H$.
\end{itemize}
\end{definition}

Suppose given a relative symplectic form $\Omega_{\pi}\in H^0(X,\wedge^2 T^*_{X/M})$ on the map $\pi$. It restricts to a symplectic form $\Omega_m\in H^0(X_m,\wedge^2 T^*_{X_m})$ on each fibre $X_m$. Note that since $T_X/\im(h)=T_{X/M}$, the relative form $\Omega_{\pi}$ can be lifted uniquely to a form $\Omega\in H^0(X,\wedge^2 T^*_{X})$ satisfying $\ker(\Omega)=\im(h)$.
 We say that the connection $h$ preserves $\Omega_{\pi}$ if for any path $\gamma\colon [0,1]\to M$  the partially-defined parallel transport maps $\PT_{\gamma}(t)\colon X_{\gamma(0)}\to X_{\gamma(t)}$ take $\Omega_{\gamma(0)}$ to $\Omega_{\gamma(t)}$.  
\begin{lemma}
\label{symp}
\begin{itemize}
\item[(i)] The connection $h$ preserves $\Omega_\pi$ precisely if $i_{v_1} i_{v_2} (d\Omega)=0$ for any two vertical vector fields $v_1,v_2\in H^0(X,T_{X/M})$.
\item[(ii)] If the connection $h$ is flat then it preserves $\Omega_\pi$ precisely if $d\Omega=0$.
\end{itemize}
 \end{lemma}
 
 \begin{proof}
 The first statement is  \cite[Theorem 4]{Gotay}. For the second, take  three vector fields $u_1,u_2,u_3$ on $X$ and consider the expression defining $d\Omega(u_1,u_2,u_3)$. We can assume that each $u_i$ is either horizontal or vertical. Since $i_h(\Omega)=0$ for any horizontal vector field $h$, and horizontal vector fields are closed under the Lie bracket, we have $d\Omega(u_1,u_2,u_3)=0$ as soon as two of the $u_i$ are horizontal. The claim then follows from (i).\end{proof}

Suppose that a discrete group $G$ acts freely and properly on  $X$ preserving the map $\pi$. Then $Y=X/G$ is a complex manifold and the quotient map $q\colon X\to Y$ is {\'e}tale. There is an induced submersion $\eta\colon Y\to M$ and a factorisation $\pi=\eta\circ q$. A connection $h\colon \pi^*(T_M)\to T_X$  will be called $G$-invariant  if $g_*\circ h=h$ for all $g\in G$. There is then an induced connection $j\colon \eta^*(T_M)\to T_Y$ on $\eta$ uniquely defined by the condition that $q_*\circ h =q^*(j)$. We say that the connection $h$ descends along the quotient map $q$.

\subsection{Pre-Joyce structures}
\label{prejoyce}

Let $M$ be a complex manifold and let $\pi\colon X=T_M\to M$ be the total space of the tangent bundle of $M$. There is  a canonical isomorphism $\nu\colon \pi^*(T_M)\to T_{X/M}$  obtained by composing the chain of identifications
\begin{equation}\pi^*(T_M)_x=T_{M,\pi(x)}=  T_{T_{M,\pi(x)},x}= T_{X_{\pi(x)},x}=T_{X/M,x},\end{equation}
 and we set $v=i\circ\nu$.  A connection $h\colon \pi^*(T_M)\to T_X$ on $\pi$ then defines a family of such connections $h_\epsilon=h+\epsilon^{-1} v$ parameterised by $\epsilon\in \bC^*$. We call such a family a $\nu$-pencil of connections.

\begin{equation*}
\xymatrix@C=1em{  
 0\ar[rr] && T_{X/M}  \ar[rr]^{i} &&T_X  \ar[rr]^{\pi_*} &&\pi^*(T_M) \ar@/_1.8pc/[ll]_{h_\epsilon} \ar@/^1.8pc/[llll]^\nu \ar[rr] && 0 } \end{equation*}

Suppose that $M$ is equipped with a holomorphic symplectic form $\omega\in H^0(M,\wedge^2 T^*_M)$. Via the isomorphism $\nu$ we obtain a  relative symplectic form $\Omega_\pi\in H^0(X,\wedge^2 T_{X/M}^*)$.  
We say that a connection on $\pi$ is symplectic if it preserves $\Omega_\pi$.  

\begin{defn}
A pre-Joyce structure on a complex manifold $M$ consists of
\begin{itemize}
\item[(i)] a {holomorphic symplectic form} $\omega$ on $M$,
\item[(ii)] a {non-linear connection} $h$ on the tangent bundle $\pi\colon X=T_M\to M$,
\end{itemize}
such that for each $\epsilon\in \bC^*$ the connection $h_\epsilon=h+\epsilon^{-1}v$ is flat and symplectic.\end{defn}

To clarify this definition we now describe it in local co-ordinates, although the resulting expressions will play no role in what follows. Given a local co-ordinate system  $(z_1,\cdots,z_n)$ on $M$ there are  associated linear co-ordinates $(\theta_1,\cdots,\theta_n)$ on the tangent spaces $T_{M,m}$  obtained by writing a tangent vector in the form $\sum_i \theta_i\cdot  {\partial}/{\partial z_i}$. We thus get  induced local  co-ordinates $(z_i,\theta_j)$ on  the space $X=T_M$. In these co-ordinates $v(\frac{\partial}{\partial z_i})=\frac{\partial}{\partial \theta_i}$.

We always assume that  the co-ordinates $z_i$ are Darboux, in the sense that
\begin{equation*}
\label{wpq}\omega = \frac{1}{2}\,\sum_{p,q} \omega_{pq} \cdot dz_p\wedge dz_q,\end{equation*}
with $\omega_{pq}$ a constant skew-symmetric matrix.  We denote by $\eta_{pq}$ the inverse matrix.

The fact that the connection $h$ is flat and symplectic ensures that we can write it in Hamiltonian form\begin{equation}
\label{above}h\Big(\frac{\partial}{\partial z_i}\Big)= \frac{\partial}{\partial z_i} + \sum_{p,q} \eta_{pq} \cdot \frac{\partial W_i}{\partial \theta_p} \cdot \frac{\partial}{\partial \theta_q},\end{equation}
for functions $W_i\colon X\to \bC$. 
 The connection $h_\epsilon$ is then flat for all $\epsilon\in \bC^*$  if we can take $W_i=\partial W/\partial \theta_i$ for a  single function $W\colon X\to \bC$, which moreover satisfies 
\begin{equation}
\label{point_intro}
\frac{\partial^2 W}{\partial \theta_i \partial z_j}-\frac{\partial^2 W}{\partial \theta_j \partial z_i }=\sum_{p,q} \eta_{pq} \cdot \frac{\partial^2 W}{\partial \theta_i \partial \theta_p} \cdot \frac{\partial^2 W}{\partial \theta_j \partial \theta_q}.\end{equation}
The function $W$ is called the Pleba{\'n}ski function, and the partial differential equations  \eqref{point_intro} are  known  as Pleba{\'n}ski's second heavenly equations  \cite{DM}. 

\subsection{Complex \hk structures}

By a complex \hk structure on  a complex manifold $X$ we mean the data of a holomorphic metric $g\colon T_X \tensor T_X \to \O_X,$
together with  endomorphisms $I,J,K\in \End_X(T_X)$ satisfying the quaternion relations
\begin{equation*}I^2=J^2=K^2=IJK=-1,\end{equation*}
which preserve $g$, and which are parallel for the  holomorphic Levi-Civita connection $\nabla$:
\begin{equation}g(R (u_1), R (u_2))= g(u,v), \qquad \nabla(R=0, \qquad R\in \{I,J,K\}.\end{equation}
Such structures have appeared before in the literature, often under different names.

Let $M$ be a complex manifold with a holomorphic symplectic form $\omega$. A non-linear connection $h$ on the tangent bundle $\pi\colon X=T_M\to M$ gives a decomposition\begin{equation}T_X=\im(v)\oplus \im(h)\isom \pi^*(T_M)\tensor_{\bC}\bC^2.\end{equation}
We can define a metric $g\colon T_X\otimes T_X\to \O_X$ by taking the tensor product of $\pi^*(\omega)$ with the standard symplectic form on $\bC^2$, and an action of the quaternions on $T_X$ by identifying the complexification of the quaternions $\bH\tensor_{\bR}\bC$  with the algebra $\End_{\bC}(\bC^2)$. With appropriate conventions this leads to the formulae\begin{equation}
\label{ijk}
\begin{aligned}
I\circ h&=i\cdot h,\qquad\quad & J\circ h&=-v,  &\qquad\quad  K\circ h&=i\cdot  v, \\
I\circ v&=-i\cdot v,   &  J\circ v&=h,   & K\circ v&=i \cdot h, \end{aligned}\end{equation}
which should be interpreted as equalities of maps $\pi^*(T_M)\to T_X$, and
\begin{equation}\label{g}g(h(u_1),v(u_2))=\tfrac{1}{2}\omega(u_1,u_2), \qquad g(h(u_1),h(u_2))=0= g(v(u_1),v(u_2)).
\end{equation}
It is easily checked that $g$ is preserved by the endomorphisms $I,J,K$.

The following result implies in particular that a pre-Joyce structure on a complex manifold $M$  induces a complex \hk structure on the total space $X=T_M$.

\begin{thm} The   endomorphisms $I,J,K$ are parallel for the Levi-Civita connection $\nabla$ associated to $g$ precisely if the connection $h_\epsilon$ is flat and symplectic for all $\epsilon\in \bC^*$.\end{thm}

\begin{proof}
We begin with a general remark. Let $g\colon T_X\times T_X\to \O_X$ be a metric on a complex manifold $X$ with associated Levi-Civita connection $\nabla$. Let $R\in \End_X(T_X)$ be an endomorphism which is compatible with $g$ and satisfies $R^2=-1$. We can then define a 2-form $\Omega$ on $X$ by setting $\Omega_R(u_1,u_2)=g(R(u_1),u_2)$. Let $H\subset T_X$ denote the $+i$ eigenbundle of $R$. Then standard proofs from K{\"a}hler geometry apply unchanged in this holomorphic context  to give implications
\begin{equation}\label{ab}\nabla(R)=0\implies [H,H]\subset H, \qquad \nabla(R)=0 \iff d \Omega_R=0.\end{equation}

Return now to  the setting above. For $\epsilon \in \bC^*$ we  introduce the endomorphism  \begin{equation}J_\epsilon=I-i\epsilon^{-1}(J+iK).\end{equation} A simple calculation using the definitions \eqref{ijk} shows that  $J_\epsilon^2=-1$, and that the $+i$ eigenbundle of $J_\epsilon$ coincides with $H_\epsilon=\im(h_\epsilon)$. 

As in Section \ref{prejoyce}, the symplectic form $\omega$ on $M$ induces a relative symplectic form $\Omega_\pi$ on the projection $\pi\colon X\to M$. Moreover, as explained before Lemma \ref{symp}, there is then a unique 2-form $\Omega_\epsilon$ on $X$ satisfying  the conditions\begin{equation}\ker(\Omega_\epsilon)=H_\epsilon, \qquad \Omega_\epsilon(v(u_1),v(u_2))=\omega(u_1,u_2),\end{equation} where $u_1,u_2$ are arbitrary vector fields on  $M$. Another calculation using  \eqref{ijk} and \eqref{g} shows that this form is given explicitly by the formula
\begin{equation}
\Omega_\epsilon =\epsilon^{-2}\cdot \Omega_++2i\epsilon^{-1}\cdot \Omega_I +\Omega_-, \qquad \Omega_\pm=\Omega_{J\pm iK}.\end{equation}

We can now prove the Theorem. Suppose first that $I,J,K$ are parallel. Then  $J_\epsilon$ is parallel for all $\epsilon\in \bC^*$, and applying \eqref{ab} with $R=J_\epsilon$ we find that $[H_\epsilon,H_\epsilon]\subset H_\epsilon$ and hence that $h_\epsilon$ is flat. Since $\Omega_\epsilon$ is also parallel and hence closed, applying Lemma \ref{symp} shows that $h_\epsilon$ is  symplectic.
Conversely suppose that for all $\epsilon \in \bC^*$ the connection $h_\epsilon$ is flat and symplectic.  Then by ~Lemma \ref{symp} again, $d\Omega_\epsilon=0$  for all $\epsilon \in \bC^*$, and this  easily implies that  $d\Omega_R=0$ for $R\in \{I,J,K\}$. By \eqref{ab} we conclude that $I,J,K$ are parallel. 
\end{proof}

%
%
%


\subsection{Period structures}

Let $M$ be a complex manifold and $\cH$ a holomorphic vector bundle on $M$. By a {lattice} in $\cH$ we mean a  locally-constant subsheaf of abelian groups $\cH^{\bZ}\subset \cH$ such that the multiplication map $\cH^{\bZ}\tensor_{\bZ}\O_M\to \cH$ is an isomorphism. There is an induced flat linear connection  $\nabla$ on $\cH$ whose  flat sections are $\bC$-linear combinations of the sections of  $\cH^{\bZ}$. 

\begin{defn}
\label{pp}
A period structure on a complex manifold $M$  consists of
\begin{itemize}
\item[(P1)] a lattice $T_M^{\bZ}\subset T_M$ whose associated flat connection we denote by $\nabla$,
\item[(P2)]  a vector field $Z\in \Gamma(M,T_M)$ satisfying $\nabla(Z)=\id$.
\end{itemize}
\end{defn}

Let $(T_M^{\bZ},\nabla,Z)$ be a period structure on a complex manifold $M$ and take a base-point $p\in M$.  A basis of the free abelian group $T_{M,p}^{\bZ}$ extends uniquely to a basis of $\nabla$-flat sections $\phi_1,\cdots,\phi_n$ of $T_M$ over a contractible open neighbourhood $p\in U\subset M$. Writing the vector field $Z$ in the form $Z=\sum_i z_i\cdot \phi_i$ then defines holomorphic functions $z_i\colon U\to \bC$, and condition (P2)  implies that $(z_1,\cdots,z_n)$ is a local co-ordinate system on $M$, and that $\phi_i=\frac{\partial}{\partial z_i}$.  Note in particular that the connection $\nabla$ is necessarily torsion-free.

Recall  that an integral affine structure on a complex manifold $M$ consists of a lattice $T_M^{\bZ}\subset T_M$ whose associated flat  connection $\nabla$ is torsion-free \cite{KS}. A local co-ordinate system $(z_1,\cdots,z_n)$ is then called integral affine if the tangent vectors $\frac{\partial}{\partial z_i}$ lie in the lattice $T_M^{\bZ}$. Such co-ordinate systems are uniquely defined up to affine transformations of the form $z_i\mapsto \sum_j a_{ij}z_j + v_i$ with $(a_{ij})\in \GL_n(\bZ)$ and $(v_i)\in \bC^n$.

Given a period structure on a complex manifold $M$ we obtain an integral affine structure  by forgetting the vector field $Z$. A system of integral affine co-ordinates $(z_1,\cdots,z_n)$ will be called integral linear if $Z=\sum_i z_i \frac{\partial}{\partial z_i}$. Such co-ordinate systems are uniquely defined up to linear transformations of the form $z_i\mapsto \sum_j a_{ij}z_j$ with $(a_{ij})\in \GL_n(\bZ)$. Thus a period structure can be thought of as an integral \emph{linear} structure. 

Using the connection $\nabla$, we can lift the vector field $Z\in \Gamma(M,T_M)$ to a horizontal vector field $E\in \Gamma(X,T_X)$.
Let us consider the case when there is a $\bC^*$ action on the manifold $M$ whose generating vector field is $Z$. 
There is a $\bC^*$ action on $X=T_M$ obtained by combining the induced action of $\bC^*$ on $X=T_M$ with the rescaling action on the fibres of $\pi\colon X=T_M\to M$ of weight $-1$. If $m_t\colon M\to M$ denotes the action of  $t\in \bC^*$  on $M$ this is the   action  for which $t\in \bC^*$ sends $v\in T_{M,m}$ to $(m_t)_*(t^{-1}v)\in T_{M,m_t(v)}$. 

\begin{lemma}
\label{jen}
The generating vector field for this $\bC^*$ action on $X$ is  the horizontal lift $E$.
\end{lemma}

\begin{proof}
If we take a system of integral linear co-ordinates $(z_1,\cdots,z_n)$ on $M$ then by definition  $Z=\sum_i z_i \cdot \frac{\partial}{\partial z_i}$.  Taking associated  co-ordinates $(z_i,\theta_j)$ on  $X=T_M$ as  before the $\nabla$-horizontal lift of $Z$ is the vector field $E=\sum_i z_i \cdot \frac{\partial}{\partial z_i}$ on $X$. The $\bC^*$ action on $X$ induced by that on $M$ is given by $(z_i,\theta_j)\mapsto (tz_i,t\theta_j)$. Composing with the contraction in the fibres we obtain the action $(z_i,\theta_j)\mapsto (tz_i,\theta_j)$ whose generating vector field is $E$.
\end{proof}
 
\begin{defn}
\label{pps}
A {period structure with skew form} on a complex manifold $M$ consists of a period structure $(T_M^{\bZ},\nabla,Z)$  together with a skew-symmetric form
\begin{equation}\eta\colon T_M^*\times T_M^*\to \O_M,\end{equation}
such that $\eta/2\pi i$  takes  integral values on the lattices $(T_M^{\bZ})^*\subset T_M^*$.
\end{defn}

The pairing $\eta$ is necessarily parallel for the flat connection $\nabla$, and it follows that it defines a holomorphic Poisson structure on $M$. 
We will be particularly interested in the case when the kernel of $\eta$ is zero. Viewing $\eta$ as a linear map $ T_M^*\to T_M$, its inverse defines a complex symplectic form
 $\omega\in H^0(M,\wedge^2 T_M^*)$.

\subsection{Joyce structures}

Let $M$ be a complex manifold with a period structure $(T_M^{\bZ},\nabla,Z)$. The rescaled lattice $(2\pi i) T_M^{\bZ}\subset T_M$ acts on $X=T_M$ by translations in the fibres. We  introduce the quotient
\begin{equation}
\label{covid}X^{\hash}=T_M^{\hash}=T_M/(2\pi i)\,T_M^{\bZ}.\end{equation}
We also consider the involution $\iota\colon X\to X$ which acts by $-1$ on the fibres of $\pi\colon X\to M$. 

\begin{defn}
\label{joyce}
Let $M$ be a complex manifold, and let $\pi\colon X=T_M\to M$ denote the total space of the holomorphic tangent bundle.  A {Joyce structure} on $M$ consists of
\begin{itemize}
\item[(a)] a period structure with skew form $(T_M^{\bZ},Z,\nabla,\eta)$ on $M$,
\item[(b)] a  pre-Joyce structure $(\omega,h)$ on $M$, \end{itemize}
satisfying the following conditions:
\begin{itemize}
\item[(J1)] the symplectic form $\omega$ is the inverse of the skew-from $\eta$,

\item[(J2)]  the connection $h$ is invariant under the action of the lattice $(2\pi i)\,T_M^{\bZ}\subset T_M$,

\item[(J3)] if $E$ is the $\nabla$-horizontal lift of the vector field $Z$, then for any vector field $v$ on $M$ \begin{equation}h([Z,v])=[E,h(v)],\end{equation}

\item[(J4)] the connection $h$ is invariant under the action of the involution $\iota\colon X\to X$.
\end{itemize}
\end{defn}



Note that once  the period structure with skew form $(T_M^{\bZ},Z,\nabla,\eta)$ on $M$ is fixed, the Joyce structure involves only one further piece of data, namely the non-linear connection $h$. For the Joyce structures appearing in this paper the  period structure is elementary and well-known, so all our work will go into defining the non-linear connection $h$.

Let us express the conditions of Definition \ref{joyce} in terms of a local co-ordinates as in Section \ref{prejoyce}. If we take a system of integral linear co-ordinates $(z_1,\cdots,z_n)$ on $M$ then by definition  $Z=\sum_i z_i \cdot \frac{\partial}{\partial z_i}$.  Taking associated  co-ordinates $(z_i,\theta_j)$ on  $X=T_M$ as  before the $\nabla$-horizontal lift of $Z$ is the vector field $E=\sum_i z_i \cdot \frac{\partial}{\partial z_i}$ on $X$. The symmetries (J2)-(J4) then translate into the following conditions on the Pleba{\'n}ski function:
\begin{gather}
\frac{\partial^2 W}{\partial \theta_p \partial \theta_q}(z_i,\theta_j+2\pi i k_j)=\frac{\partial^2 W}{\partial \theta_q \partial \theta_q}(z_i, \theta_j)\,  \\
W(\lambda z_i,\theta_j)=\lambda^{-1}  W(z_i,\theta_j),\\
W(z_i,-\theta_j)=-W(z_i,\theta_j).,\end{gather}
where $(k_1,\cdots,k_n)\in \bZ^n$ and $\lambda\in \bC^*$.

The construction we describe in this paper produces what we shall call a meromorphic Joyce structure. This means that 
the connection $h$ has poles on certain subsets of  $X$. More precisely,  there is an effective divisor $D\subset X$, and  $h$ is defined by a bundle map $h\colon \pi^*(T_M)\to T_X(D)$ satisfying \begin{equation}(\pi_*\tensor \O_X(D))\circ h=1_{\pi^*(T_M)}\tensor s_D,\end{equation}
 where  $s_D\colon \O_X\to \O_X(D)$ is the canonical inclusion. This means that when expressed in  terms of local co-ordinates as above, the  Pleba{\'n}ski function $W(z_i,\theta_j)$ is  a meromorphic function. 
 

\section{The period structure on the space of quadratic differentials}
\label{three}

For the rest of the paper we fix an integer $g>1$ and a level $\ell>2$.  As discussed in Section \ref{th}  we  insist that all curves $C$ are equipped with a level $\ell$ structure, although we suppress this from the notation. In this section we introduce the space $M=\cM(C,Q)$ which will form the base of our Joyce structure. It parameterises pairs $(C,Q)$ consisting of a smooth projective curve $C$ of genus $g$ equipped with a level $\ell$ structure,  and    a quadratic differential $Q\in H^0(C,\omega_C^{\tensor 2})$ with simple zeroes.   Any such pair $(C,Q)$ determines a branched double cover  $p\colon \Sigma\to C$ which we call the spectral curve. We construct a period structure with skew form on $M$, and give a moduli-theoretic description of the fibres of the map \eqref{covid} in terms of line bundles with connection on $\Sigma$. 

\subsection{Moduli space of quadratic differentials}
\label{quad}

Let us begin by  recalling the definition of a level structure. 
Given a smooth complex projective curve $C$ of genus $g$, the homology group $H_1(C,\bZ/\ell)$ is a free $\bZ/\ell$-module of rank $2g$. The intersection form defines a non-degenerate skew-symmetric  form \begin{equation}\<-,-\>\colon H_1(C,\bZ/\ell)\times H_1(C,\bZ/\ell)\to \bZ/\ell.\end{equation} A level $\ell$ structure on $C$ is a basis  $(\alpha_1,\cdots,\alpha_g,\beta_1,\cdots, \beta_g)$ for  $H_1(C,\bZ/\ell)$ which is symplectic, in the sense that $\<\alpha_i,\alpha_j\>=0=\<\beta_i,\beta_j\>$ and $\<\alpha_i,\beta_j\>=\delta_{ij}$.

Let $\cM(C)$ denote the moduli space of smooth complex projective curves of genus $g$ equipped with a level $\ell$ structure. As we recall in Appendix \ref{moduli}, given our  assumptions $g>1$ and $\ell>2$  this is a smooth  quasi-projective  complex variety of dimension $3g-3$.

The tangent space to $\cM(C)$ at a curve $C$  is the cohomology group $H^1(C,T_C)$, and Serre duality gives $H^0(C,\omega_C^{\tensor 2})=H^1(C,T_C)^*$, so the cotangent bundle $T^*_{\cM(C)}$ parameterises pairs $(C,Q)$ consisting of a smooth complex projective curve $C$ equipped with a level $\ell$ structure, together with an element $Q\in H^0(C,\omega_C^{\tensor 2})$. We define $\cM(C,Q)\subset T^*_{\cM(C)}$ to be the open subset of pairs $(C,Q)$ for which  $Q$ has simple zeroes. Then $M=\cM(C,Q)$ is a smooth quasi-projective  complex variety of dimension $6g-6$.

\subsection{Spectral curve}
\label{cl}

Let $C$ be a smooth complex projective curve of genus $g$, and $Q\in H^0(C,\omega_C^{\tensor 2})$ a quadratic differential with simple zeroes. The spectral curve $\Sigma$ associated to the pair $(C,Q)$ is the smooth projective curve cut out inside the total space of the cotangent bundle $T^*_C$ by the equation $y^2=Q$. The projection  $(x,y)\mapsto x$ defines a branched double cover $p\colon \Sigma\to C$, and there is a covering involution $\sigma\colon \Sigma\to \Sigma$ defined by $(x,y)\mapsto (x,-y)$. The assumption that $Q$ has simple zeroes ensures that $\Sigma$ is smooth, and the fact that $Q$ has at least one zero implies that $\Sigma$ is connected.

Denote by  $R\subset \Sigma$ the branch divisor of  the  map $p\colon \Sigma\to C$.  It has degree $4g-4$, so by the Riemann-Hurwitz formula the spectral curve $\Sigma$ has genus $4g-3$. The dual of the derivative of $p$ defines a canonical section $s\colon p^*(\omega_C)\to \omega_\Sigma$ fitting into  a short exact sequence
\begin{equation}\label{nik}0\lra p^*(\omega_C)\lRa{s} \omega_{\Sigma}\lra \O_R\lra 0.\end{equation}
On the other hand the square-root of $p^*(Q)$ defines a section of $p^*(\omega_C)$  with simple zeroes on $R$, and hence a short exact sequence
\begin{equation}\label{nik2}0\lra \O_\Sigma\lRa{\phi} p^*(\omega_C)\lra \O_R\lra 0.\end{equation}

We define the invariant and anti-invariant homology groups \begin{equation*}\label{subgroups}H_1(\Sigma,\bZ)^{\pm}=\{\gamma\in H_1(\Sigma,\bZ): \sigma^*(\gamma)=\pm \gamma\},\end{equation*}
and similarly for the cohomology groups $H^1(\Sigma,\bZ)^\pm$, $H^1(\Sigma,\bC)^\pm$, etc. There is a short exact sequence of free abelian groups
\begin{equation}
\label{ext}0\lra H_1(\Sigma,\bZ)^-\lra H_1(\Sigma,\bZ)\lRa{p_*} H_1(C,\bZ)\lra 0,\end{equation}
the map $p_*$ being surjective\todo{Explain} because $p$ is ramified.

Taking maps of \eqref{ext} into $\bZ$  shows that the image of $p^*\colon H^1(C,\bZ)\to H^1(\Sigma,\bZ)$ coincides with the subgroup $H^1(\Sigma,\bZ)^+$.
The anti-invariant homology group $H_1(\Sigma,\bZ)^-$ is  therefore free  of rank $6g-6$. 
We also consider the extended  group
\begin{equation*}\tilde{H}^1(\Sigma,\bZ)^-:=\Hom_{\bZ}(H_1(\Sigma,\bZ)^-,\bZ)=H^1(\Sigma,\bZ)/H^1(C,\bZ).\end{equation*}

\subsection{Period structure}

Introduce the vector bundle $\cH\to M$ whose fibre over a point $(C,Q)$ is the vector space $H^1(\Sigma,\bC)^-$. 
It contains a lattice $\cH^{\bZ}\subset \cH$ whose fibres are the groups $\tilde{H}^1(\Sigma,\bZ)^-$.  The associated flat connection $\nabla^{\GM}$ on $\cH$ is the Gauss-Manin connection. The dual bundle $\cH^*$ has fibres $H_1(\Sigma,\bC)^-$ and contains the dual lattice $(\cH^\bZ)^*$ with fibres $H_1(\Sigma,\bZ)^-$. The intersection form   defines a    parallel  skew-symmetric form  on $\cH^*$ 
which takes integral values on $(\cH^{\bZ})^*$.

For each point $(C,Q)\in M$, the tautological  1-form $y \,dx$ on $T^*C$ restricts to a 1-form $\lambda\in H^0(\Sigma,\omega_\Sigma)$  satisfying  $\lambda^{\tensor 2}=p^*(Q)$ and $\sigma^*(\lambda)=-\lambda$. This  should not be confused with the section $\phi$ appearing in \eqref{nik2}: there is a  relation $\lambda=s\circ \phi$.
The de Rham cohomology class associated to $\lambda$  is an element of $H^1(\Sigma,\bC)^-$, and the resulting map  \begin{equation}
\label{delta}\delta\colon M\to \cH, \qquad (C,Q)\mapsto [\lambda]\in H^1(\Sigma,\bC)^-\end{equation} is a holomorphic section of the bundle $\cH$.

\begin{thm}[\cite{Veech}]\label{existperiods} The  covariant derivative of  $\delta$ with respect to the Gauss-Manin connection defines an isomorphism\todo{Check/find reference for this statement}  $\nabla^{\GM}(\delta)\colon T_{M} \to \cH$.
\end{thm}

Taking  a basis $(\gamma_1,\cdots,\gamma_{n})\subset H_1(\Sigma,\bZ)^-$ at some point $(C,Q)\in M$ and extending to nearby points using the Gauss-Manin connection gives locally-defined  functions on $M$ \begin{equation}
\label{periodcoords}z_i=Z(\gamma_i)=([\lambda],\gamma_i)=\int_{\gamma_i} \sqrt{Q}, \qquad 1\leq i\leq n.\end{equation}
 Theorem \ref{existperiods} is then  the statement that these functions are local co-ordinates on $M$.  Note that the associated linear co-ordinates $\theta_i=(dz_i,-)$ on the fibres of the bundle $T_M\to M$ considered in Section \ref{prejoyce} correspond, under the isomorphism of Theorem \ref{existperiods}, to the functions on the fibres  of the bundle $\cH\to M$ given by pairing with the classes $\gamma_i$.

We can use the isomorphism of Theorem \ref{existperiods} to transfer the data $(\cH^{\bZ},\nabla^{\GM}, \delta)$ from the bundle $\cH$ to the tangent bundle $T_M$. This defines a period structure  $(T_M^{\bZ},\nabla, Z)$ on $M$ for which the periods \eqref{periodcoords} are integral linear co-ordinates. The required identity  $\nabla(Z)=\id$ holds by definition. The intersection form on $\cH^*$ induces  a skew-symmetric form $ T^*_M\times T^*_M\to \O_M$ which takes integral values on the lattice $(T_M^{\bZ})^*\subset T_M^*$. We obtain a period structure with skew form by taking $\eta$ to be this form multiplied by $2\pi i$. Since the  intersection form is non-degenerate, the inverse to $\eta$ defines a symplectic form $\omega\in H^0(M,\wedge^2 T^*_M)$. 

\subsection{Prym variety and abelian connections}
\label{prym}

We denote by $J(C)$ and $J(\Sigma)$ the Jacobians of the curves $C$ and $\Sigma$.  Set 
\begin{equation*}J(\Sigma)^-=\{M\in J(\Sigma): M\tensor \sigma^*(M)\isom \O_{\Sigma}\},\qquad J^2(C)=\{P\in J(C):P^{\tensor 2}\isom \O_C\}.\end{equation*}
The pullback map $p^*\colon J(C)\to J(\Sigma)$ is injective \cite[Section 3]{mum}, and we identify $J(C)$ with its image. The Prym variety  is defined by either of the two quotients
\begin{equation}\label{canc}P(\Sigma)=J(\Sigma)/J(C) = J(\Sigma)^-/ J^2(C).\end{equation}

To see that the  two quotients in \eqref{canc} are indeed the same, consider the map $j\colon J(\Sigma)^-\to J(\Sigma)/J(C)$ induced by the inclusion $J(\Sigma)^-\subset J(\Sigma)$. Then $j$ is surjective because for any $M\in J(\Sigma)$ we can write $M=N^{\tensor 2}$, and then \begin{equation*}
\label{hampers}M=(N\tensor \sigma^*(N^*))\tensor (N\tensor \sigma^*(N)).\end{equation*} The first factor clearly lies in $J(\Sigma)^-$, and it is proved\todo{Check this} in \cite[Section 3]{mum} that the second lies in $J(C)$. The kernel of $j$ is the intersection $J(C)\cap J(\Sigma)^-\subset J(\Sigma)$, and since any element $M\in J(C)$ satisfies $\sigma^*(M)=M$, this coincides with $J^2(C)$.

We also consider the spaces  $J^\hash(C)$ and $J^\hash(\Sigma)$  of line bundles with connection. We can again identify $J^\hash(C)$ with the image of the pullback map $p^*\colon J^\hash(C)\to J^\hash(\Sigma)$. We set
\begin{equation*}J^\hash(\Sigma)^-=\{(M,\partial)\in J^\hash(\Sigma): (M,\partial)\tensor \sigma^*(M,\partial)\isom (\O_{\Sigma},d)\}.\end{equation*}
Note that if 
 $(N,\partial_N)$ is a line bundle with connection on $C$, and  $P^{\tensor 2}\isom N$, then there is a  unique connection $\partial_P$  such that $(P,\partial_P)^{\tensor 2}\isom (N,\partial_N)$. In particular each 
 $P\in J^2(C)$ has a unique connection $\partial_P$ satisfying $(P,\partial_P)^{\tensor 2}=(\O_C,d)$, and we can therefore identify $J^2(C)$ with a subgroup of $J^\hash(\Sigma)$. We  define \begin{equation}\label{bed}P^{\hash}(\Sigma)= J^\hash(\Sigma)/J^\hash(C) = J^{\hash}(\Sigma)^-/J^2(C),\end{equation}
 with a similar argument as before showing that these two quotients are equal.\todo{Small missing detail here.} 
 
 The exact sequence \eqref{ext} shows that
 \begin{equation*}
 \label{idiotputin}\tilde{H}^1(\Sigma,\bC^*)^-:=\Hom_{\bZ}(H_1(\Sigma,\bZ)^-,\bC^*)\isom 
 H^1(\Sigma,\bC^*)/H^1(C,\bC^*).\end{equation*}
The Riemann-Hilbert isomorphism $J^\hash(\Sigma)\isom H^1(\Sigma,\bC^*)$ then induces an isomorphism
 \begin{equation}
 \label{mapp}P^{\hash}(\Sigma)\isom \tilde{H}^1(\Sigma,\bC^*)^-.\end{equation}

\subsection{Anti-invariant branched connections}
\label{aunty}

Let $F$ be a vector bundle on the spectral curve $\Sigma$. By a branched connection on  $F$ we mean a meromorphic connection  $\partial \colon F\to F\tensor \omega_\Sigma(R)$ with simple poles on the branch divisor $R\subset \Sigma$. The line bundle $\O_\Sigma(R)$ has a canonical branched connection $d\colon \O_{\Sigma}(R)\to \O_\Sigma(R)\tensor \omega_{\Sigma}(R)$ induced by the de Rham differential applied to functions on $\Sigma$ with  poles on $R$. This connection has a simple pole with residue $-1$ at each point of $R$.

The short exact sequence \eqref{nik2} induces an isomorphism $p^*(\omega_C)\isom \O_\Sigma(R)$. We therefore obtain   a canonical branched connection $d_\phi$ on $p^*(\omega_C)$.  It is uniquely defined by the condition that $\phi$ is a flat map of bundles  with meromorphic connection $(\O_{\Sigma},d)\to (p^*(\omega_C),d_\phi)$.
We say that a line bundle with branched connection $(L,\partial)$ on $\Sigma$   is  anti-invariant if  \begin{equation}\label{cancer}(L,\partial_L)\tensor \sigma^*(L,\partial_L)\isom (p^*(\omega_C),d_\phi).\end{equation}
 It follows that $\partial$ has a simple pole  with residue $-\tfrac{1}{2}$ at each point of $R$.

Let $J_{\br}^\hash(\Sigma)$ denote the space of line bundles $L$ on $\Sigma$ equipped with anti-invariant branched connections $\partial$. The group $J^2(C)\subset J^\hash(\Sigma)$ acts on this space by tensor product, and in analogy with \eqref{bed} we define \begin{equation*}\label{brrr}
P^\hash_{\br}(\Sigma)=J_{\br}^\hash(\Sigma)/J^2(C).\end{equation*}

\begin{lemma}
\label{lemmaabove}
There is a canonical isomorphism $P^\hash(\Sigma)\isom P^\hash_{\br}(\Sigma)$.
\end{lemma}

\begin{proof}
Tensor product gives $J_{\br}^\hash(\Sigma)$ the structure of a torsor over $J^\hash(\Sigma)^-$, so  choosing a point $(L_0,\partial_0)\in J_{\br}^\hash(\Sigma)$  gives a non-canonical identification of the two spaces \begin{equation}
\label{ukraine}(L,\partial)\in J^\hash(\Sigma)\mapsto (L,\partial)\tensor (L_0,\partial_0)\in J^\hash_{\br}(\Sigma)^-,\end{equation} which descends to the quotients by $J^2(C)$. 
 To obtain a canonical bijection, take a spin structure $\omega_C^{1/2}$  on  $C$ and let $\partial_0$  be the unique branched connection on $L_0=p^*(\omega_C^{1/2})$ satisfying $(L_0,\partial_0)^{\tensor 2}=(p^*(\omega_C),d_\phi)$. Since $\omega_C^{1/2}$ is uniquely defined up to the action of $J^2(C)$, the resulting isomorphism $P^\hash(\Sigma)\isom P^\hash_{\br}(\Sigma)$ is canonically defined.  \end{proof}


\section{The spectral correspondence}
\label{four}

Let us fix a smooth projective curve $C$ of genus $g$ and a quadratic differential $Q\in H^0(C,\omega_C^{\tensor 2})$ with simple zeroes. Let $p\colon \Sigma\to C$ be the associated spectral curve with its covering involution $\sigma$. There is a very well known correspondence relating $\SL_2(\bC)$ Higgs bundles on $C$ to anti-invariant line bundles  on $\Sigma$. In this section we show how to extend this construction to include connections. The existence of this extension seems to be little known, although it is discussed by Donagi and Pantev in  \cite[Section 3.2]{DP} and also appears in a paper of Arinkin \cite{A}.

\subsection{Definition}
\label{works}

The  $\SL_2(\bC)$ spectral  correspondence \cite[Section 8]{Hitchin}, \cite[Section 3]{BNR} defines a bijection between
\begin{itemize}
\item[(i)]  rank 2 vector bundles $E$ on $C$, with $\det(E)\isom \O_C$, equipped with a  Higgs field $\Phi\colon E\to E\tensor \omega_C$ with $\operatorname{tr}(\Phi)=0$ and $\tfrac{1}{2}\tr(\Phi^2)=Q$,
\item[(ii)]line bundles $L$ on $\Sigma$ satisfying $L\tensor \sigma^*(L)\isom p^*(\omega_C)$.
\end{itemize}
The line bundle $L$ is obtained from  the eigen-decomposition of the pullback of the Higgs field $p^*(\Phi)$. In the reverse direction, the
  line bundle $L$  is sent  to the bundle $E=p_*(L)$ equipped with the Higgs field $\Phi$ which is the push-forward of the map $1\tensor \phi\colon L\to L\tensor p^*(\omega_C)$. 
  
We shall need an extension of this correspondence involving connections on the bundles $L$ and $E=p_*(L)$. More precisely,  the extension relates
\begin{itemize}
\item[(i)] connections $\nabla$ on $E$ inducing the trivial connection  on $\det(E)\isom \O_C$,
\item[(ii)] anti-invariant branched connections  $\partial$ on the line bundle $L$.
\end{itemize}
We do not claim that this correspondence is a bijection\todo{Give counter-example} in general, but it does define a birational map of the relevant moduli spaces: see Theorem \ref{birational} below.

The  extended correspondence is defined as follows. Take a bundle $E=p_*(L)$ on $C$ and a connection $\nabla$ on $E$ as in (i).
The  natural transformation $p^*p_*(L)\to L$ gives rise to a short exact sequence
\begin{equation}
\label{ses}0\lra p^*(E)\lRa{f} L\oplus \sigma^*(L) \lra \O_R\lra 0.\end{equation}
Since the map $f$ is an isomorphism away from the divisor $R\subset \Sigma$ there is a unique meromorphic connection $\tilde{\nabla}$ on  $L\oplus \sigma^*(L)$ with poles along $R$  such that the map $f$ becomes a map of bundles with meromorphic connection $p^*(E,\nabla)\to (L\oplus \sigma^*(L),\tilde{\nabla})$. 
The local calculation in the next subsection shows that the poles of $\tilde{\nabla}$ at the points of $R$ are simple. Taking the component of $\tilde{\nabla}$ along $L$ then gives the required branched connection $\partial$.

To prove that $(L,\partial)$ satisfies the identity \eqref{cancer}, take determinants of \eqref{ses} to get a map $\det(f)\colon \O_{\Sigma}\to L\tensor\sigma^*(L)$. The support of the cokernel is precisely $R$ and it  follows for degree reasons  that $\coker(\det(f))=\O_R$. The sequence \eqref{nik2} then shows that we can identify $L\tensor\sigma^*(L)$ with $p^*(\omega_C)$ in such a way that the map $\det(f)$ coincides with the map $\phi$. But then $\det(\tilde{\nabla})=\partial\tensor 1 + 1\tensor \sigma^*(\partial)$ is related to the trivial connection $\det(p^*(\nabla))$  by the meromorphic gauge change $\phi$, and hence coincides with $d_\phi$. 

\subsection{Local computation at branch-point}
\label{cumc}

Consider a $\sigma$-invariant neighbourhood $U\subset \Sigma$ of a branch-point $p\in R$. Choose a local co-ordinate $w\colon U\to \bC$ satisfying $\sigma^*(w)=-w$, and hence $w(p)=0$. Choose also a local non-vanishing section $s\in H^0(U,L)$. In terms of the basis of sections $(s,\sigma^*(s))$  we can write the induced connection $\tilde{\nabla}$ on $L\oplus\sigma^*(L)$ in the form
\begin{equation*}\tilde{\nabla}= d + \mat{\alpha(w)}{\beta(w)}{\gamma(w)}{\delta(w)}dw,\end{equation*}
where $\alpha,\beta,\gamma,\delta\colon U\to \bC$ are meromorphic functions defined in a neighbourhood of  $0\in \bC$, and regular away from $0$. The invariance of the pullback connection on $p^*(E)$ implies that $\gamma(w)=-\beta(-w)$ and $\delta(w)=-\alpha(-w)$.

The sequence \eqref{ses} shows that a section of $p^*(E)$ over $U$ is determined by sections of $L$ and $\sigma^*(L)$ which agree at the branch-point $p$. Since derivatives of regular sections of  $p^*(E)$ are also regular  sections of $p^*(E)$ we can write
\begin{equation}
\label{aone}\tilde{\nabla}_{\frac{\partial}{\partial w}}\vect{1}{1}= \vect{\alpha(w)+\beta(w)}{-\alpha(-w)-\beta(-w)}=\vect{c_+}{c_+}+O(w), \end{equation}\begin{equation}\label{btwo}\tilde{\nabla}_{\frac{\partial}{\partial w}}\vect{w}{-w}=\vect{1}{-1} +w\vect{\alpha(w)-\beta(w)}{\alpha(-w)-\beta(-w)}=\vect{c_-}{c_-}+O(w),\end{equation}
with $c_\pm \in \bC$. It follows that $\alpha(w)$ and $\beta(w)$ have at worst simple poles at $w=0$ and we can therefore  write
\begin{equation*}\alpha(w)=-\frac{a}{2w}+ c + O(w), \qquad \beta(w)=\frac{b}{2w}-d + O(w),\end{equation*}
with $a,b,c,d\in \bC$. Equation \eqref{aone} then implies that $b=a$ and $d=c$, and \eqref{btwo} implies that $a+b=2$.  Thus
\begin{equation*}\alpha(w)=-\frac{1}{2w}+ c + O(w), \qquad \beta(w)=\frac{1}{2w}-c + O(w).\end{equation*}
for some element $c\in \bC$. In particular $\tilde{\nabla}$ has simple poles on $R$. The induced branched connection on $L$ is given  by $\partial = d+\alpha(w) dw$. 

\subsection{Explicit description}
\label{explicit}

It will be useful later to have a more detailed description of  the extended spectral correspondence described in Section \ref{works}.

%

The functor $p_*\colon \Coh(\Sigma)\to \Coh(C)$ has a left adjoint $p^*$ and a right adjoint $p^!$. They are related by $p^!(-)=p^*(-)\tensor \omega_{\Sigma/C}$, where $\omega_{\Sigma/C}=\omega_{\Sigma}\tensor p^*(\omega_C^{\vee})$ is the relative dualising sheaf. The short exact sequence \eqref{nik}  induces an identification of $\omega_{\Sigma/C}$ with $\O_{\Sigma}(R)$. We can then identify $p^!$ with  the functor $p^*(-)\tensor \O_{\Sigma}(R)$.
  There are  natural transformations \begin{equation}\eta\colon p^*\circ p_*\to 1, \qquad \chi\colon 1\to p^!\circ p_*.\end{equation}  
 

Let $L$ be a  line bundle on $\Sigma$. Setting $f=(\eta_L,\eta_{\sigma^*(L)})$ and $g=(\chi_L,\chi_{\sigma^*(L)})$ gives maps
\begin{equation}p^*(E)\lRa{f} L\oplus \sigma^*(L)\lRa{g} p^*(E)\tensor \O_{\Sigma}(R).\end{equation}
We claim that they are mutually inverse away from the ramification divisor:

\begin{lemma}
Let $s_R\colon \O_{\Sigma}\to \O_{\Sigma}(R)$ denote the canonical inclusion. Then
\begin{equation}g\circ f = 1_{p^*(E)}\tensor s_R, \qquad f\tensor \omega_{\Sigma/C}\circ g = 1_{L\oplus \sigma^*(L)} \tensor s_R.\end{equation}
\end{lemma}

\begin{proof}
As in the previous section we can view local sections of $p^*(E)$ as consisting of pairs of local sections $u$ of $L$ and $v$ of $\sigma^*(L)$ whose restrictions to the branch divisor  $R$ coincide. Then $\eta_L\colon p^*(E)\to L$ sends such a pair   to the local section $u$, and  $\chi_L(-R)\colon L(-R)\to p^*(E)$ sends a local section $u$ of $L$ which vanishes on $R$ to the  pair $(u,0)$.  The result follows. \end{proof}

The extended correspondence  is defined by viewing the maps $f$ and $g$ as meromorphic gauge transformations, and using them to transfer the connection $p^*(\nabla)$ from $p^*(E)$ to $L\oplus\sigma^*(L)$. We then take the first component  to obtain a meromorphic connection $\partial $ on $L$ with poles along $R$.
The connection $\partial$  is therefore given by the composite map
\begin{equation}L\lRa{\chi_L} p^*(E)\tensor \O_{\Sigma}(R)\lRa{p^*(\nabla)\tensor 1 +1\tensor d} p^*(E)\tensor \O_{\Sigma}(R)\tensor \omega_{\Sigma}(R) \lRa{ \eta_L\tensor \omega_{\Sigma}(2R)} L\tensor \omega_{\Sigma}(2R).\end{equation}
Here $d$ denotes the canonical branched connection on $\O_{\Sigma}(R)$ induced by the de Rham differential, as in Section \ref{aunty}. 
 Note that although  the resulting connection $\partial$ {\it a priori} has double poles along $R$, the local calculation in Section \ref{cumc} shows that these poles are in fact of order one.


\section{Two diagrams of moduli spaces}
\label{five}

In this section we introduce a diagram of moduli spaces which will play a key role in our construction of the Joyce structure on the space $\cM(C,Q)$. These moduli spaces parameterise smooth projective curves $C$ of genus $g$ equipped with various extra structures involving vector bundles, Higgs fields and connections. We defer discussion  of the existence and basic properties of the moduli spaces to Appendix \ref{moduli}. The main result of this section focuses on the two most interesting maps in our diagram: $\alpha$ and $\beta_\epsilon$. We show that the first is birational and the second is generically {\'e}tale.

\subsection{The diagrams}

We use the following notation:
\begin{itemize}
\item $C$ is a  complex projective curve of genus $g$, equipped with a level $\ell$ structure,
\item $Q\in H^0(C,\omega_C^{\tensor 2})$ is a quadratic differential on $C$ with simple zeroes,
\item $p\colon \Sigma\to C$ is the spectral curve defined by the quadratic differential $Q$,
\item $E$ is a stable rank 2 vector bundle on $C$ with trivial determinant,
\item $\Phi$ is a trace-free Higgs field on $E$ such that $\tfrac{1}{2}\tr(\Phi^2)$ has simple zeroes,
\item $\nabla$  is a connection on $E$ inducing the trivial connection on $\det(E)$,
\item $L$ is a line bundle on $\Sigma$ such that $L\tensor \sigma^*(L)\isom p^*(\omega_C)$ and $p_*(L)$ is stable,
\item $\partial$ is an anti-invariant branched connection on $L$.
\end{itemize} 

\begin{equation}
\begin{gathered}\label{biggy}
\xymatrix@C=.9em{  & \cM(C,E,\nabla,\Phi) \ar[dl]_{\alpha} \ar[dr]^{\beta_\epsilon}\\
\cM(C,Q,L,\partial)  \ar[d]_{\pi_3} && \cM(C,Q,E,\nabla)\ar[d]_{\pi_2} \ar@{->}[rrr]^{\rho'} &&&\cM(C,E,\nabla)\ar[d]_{\pi_1}   \\
\cM(C,Q) \ar@{<->}[rr]^{=} && \cM(C,Q)\ar@{->}[rrr]^{\rho}&&& \cM(C)
}\end{gathered}
 \end{equation}
 
Let us fix a parameter $\epsilon\in \bC^*$ and contemplate the diagram \eqref{biggy}. Each moduli space 
 parameterises the indicated objects,  and the  maps $\rho, \rho'$ and $\pi_i$ are the obvious projections. Note that $\cM(C)$ and $\cM(C,Q)=M$ are the moduli spaces appearing in Section \ref{quad}.
The map $\alpha$ is the extended spectral correspondence discussed in the previous subsection, and $\beta(\epsilon)$ is defined by the rule
\begin{equation*}\beta_\epsilon(C,E,\nabla,\Phi)=(C,\tfrac{1}{2}\tr(\Phi^2), E, \nabla-\epsilon^{-1}\Phi).\end{equation*}
We refer the reader to Appendix \ref{moduli} for further details on the moduli spaces appearing in  \eqref{biggy}. Given our standing assumptions $g>1$ and $\ell>2$, all the spaces appearing are smooth quasi-projective complex varieties which co-represent the relevant moduli functors.  

%

Consider the bundle $\cH^{\hash}=\cH/(2\pi i)\cH^{\bZ}$ over  $M$ with  fibres  $\tilde{H}^1(\Sigma,\bC^*)^-$. Under the isomorphism  of Theorem \ref{existperiods} it corresponds to the quotient \eqref{covid}. We now consider a second diagram of spaces \eqref{seconddiag} which can be attached to the left-hand side of \eqref{biggy}.

\begin{equation}
\label{seconddiag}
\begin{gathered}
\xymatrix@C=1em{  
 T_{M}^{\hash}\ar[rrrr]^{\nabla^{\GM}(\delta)}\ar[d]_{\pi_5}&&&& \cH^{\hash}\ar[d]_{\pi_4}&&& \cM(C,Q,L,\partial)\ar[d]_{\pi_3} \ar[lll]_{\tau\ \ }  \\
M \ar@{<->}[rrrr]^{=} &&&& M\ar@{<->}[rrr]^{=}  &&& \cM(C,Q)
}
\end{gathered}\end{equation}

Note that the fibre of the map $\pi_3$ over a pair $(C,Q)$ is the open subset of the space $J_{\br}^{\hash}(\Sigma)$  defined by the condition that $p_*(L)$ is stable. The map $\tau$ is then given on fibres by the composite of the quotient map $J_{\br}^{\hash}(\Sigma)\to P^\hash_{\br}(\Sigma)$ with the isomorphism of Lemma \ref{lemmaabove} and the Riemann-Hilbert isomorphism  \eqref{mapp}. It is an {\'e}tale map of complex manifolds.

\subsection{Generic finiteness results}
\label{finite}

The following result contains the two main non-trivial facts we will need for the construction of our Joyce structure.

\begin{thm} \label{birational}Fix a genus $g>1$ and a level $\ell>2$. Then\begin{itemize}
\item[(i)] the map $\alpha$ is birational,
\item[(ii)] the map $\beta_\epsilon$ is generically {\'e}tale.\todo{
There is an involution $\iota$ of the space $\cM(C,Q,L,\partial)$ which fixes $(C,Q)$ and sends $(L,\partial)$ to $\sigma^*(L,\partial)$. A possible strengthening of Theorem \ref{birational} states that  the maps $\alpha,\beta_\epsilon$ induce a birational isomorphism $\beta_\epsilon\circ \alpha^{-1}\colon \cM(C,Q,L,\partial)/\<\iota\>\dashrightarrow \cM(C,Q,E,\nabla_\epsilon)$.}
\end{itemize}\end{thm}

\begin{proof} Note first that the sources and targets of the maps $\alpha$ and $\beta_\epsilon$ are smooth quasi-projective varieties of the same dimension. By generic smoothness and the dimension theorem, for (i)  it will be enough to prove that the general fibre of $\alpha$  is a single point, and for (ii) that the general fibre of $\beta_\epsilon$ is finite, or equivalently, that $\beta_\epsilon$ is dominant.

For (i), suppose we have two connections on $E$ giving rise to the same connection on $L$. In terms of the local computation of Section \ref{cumc} this means that we have two possible $\beta_i$ with the same  $\alpha$, and in particular, the same $c\in \bC$. Then the difference $\beta_2(w)-\beta_1(w)$ is regular and vanishes at the branch-point $w=0$. Globally the difference $(\beta_2(w)-\beta_1(w))dw$ corresponds to a section $\sigma^*(L)\to L\tensor \omega_{\Sigma}$. 
But since $\Sigma$ has genus $4g-3$, and $R$ has degree $4g-4$ 
\begin{equation}\chi(\sigma^*(L),L\tensor \omega_{\Sigma}(-R))=\chi(\omega_{\Sigma}(-R))=(4g-3)-1-(4g-4)=0,\end{equation} and it follows that for generic $L$ any such section is zero.\todo{Small issue here.} 

For (ii), note first that by the relation $\nabla_\epsilon=\nabla-\epsilon^{-1}\Phi$, the source of $\beta_\epsilon$ may be equivalently viewed as paramaterising $C,E,\Phi$ and $\nabla_\epsilon$, and therefore $\beta_\epsilon$ is a base-change of the map
\begin{equation}\gamma\colon \cM(C,E,\Phi)\to \cM(C,E,Q)\end{equation} defined by $Q=\tfrac{1}{2}\tr(\Phi^2)$. That this map is dominant follows from the proof of \cite[Theorem 1]{BNR}. Namely, we first extend $\gamma$ by dropping the condition that $Q$ has simple zeroes, and then show that the resulting map is dominant. But by the results of \cite{DL}, for each curve $C$ there exists a bundle $E$ for which the fibre of $\gamma$  over the point $(C,E,0)$  is a single point. \end{proof}

Note that the proof of Theorem \ref{birational} actually gives more: the open locus over which the map $\alpha$ is birational intersects each fibre $\pi_3^{-1}(m)$, and similarly, the open locus over which the map $\beta_\epsilon$ is {\'e}tale intersects each fibre $\pi_2^{-1}(m)$. We record this as

\begin{cor}
\label{cor}
Let $C$ be a smooth projective curve of genus $g>1$, and $Q\in H^0(C,\omega_C^{\tensor 2})$ a quadratic differential with simple zeroes. Let $p\colon \Sigma\to C$ be the associated spectral curve. Then for each $\epsilon\in \bC^*$, the composite $\beta_\epsilon\circ\alpha^{-1}$ defines a dominant rational map
\begin{equation}\gamma_\epsilon\colon \cM_\Sigma(L,\partial)\dashrightarrow \cM(E,\nabla)\end{equation}
between the moduli space of anti-invariant branched connections on $\Sigma$, and the moduli space of rank 2 connections on $C$ with trivial determinant. \qed \end{cor}


\section{Symplectic forms and their preservation}
\label{six}

In this section we introduce the relevant symplectic forms on the moduli spaces and show that our maps preserve them.

\subsection{The Atiyah-Bott symplectic form} 

Let $E$ be a bundle on $C$ and $\nabla\colon E\to E\tensor \omega_C$ a connection. There is an associated de Rham complex
\begin{equation}0\lra \lEnd_{\O_C}(E)\lRa{\nabla} \lEnd_{\O_C}(E)\tensor \omega_C\lra 0.\end{equation}
We denote by $\bH^i_{dR}(C,\nabla)$ the $i$th hypercohomology of this complex. The long exact sequence in hypercohomology gives relations
\begin{equation}\bH^0_{dR}(C,\nabla)= H^0(C,\lEnd_{\O_C}(E)),\quad \bH^2_{dR}(C,\nabla)= H^1(C,\lEnd_{\O_C}(E)\tensor \omega_C),\end{equation}
and a short exact sequence
\begin{equation}
\label{snow}0\lra H^0(C,\lEnd_{\O_C}(E)\tensor \omega_C)\lra \bH^1_{dR}(C,\nabla)\lra H^1(C,\lEnd_{\O_C}(E))\lra 0.\end{equation}

Consider  the moduli space $\Flat_C(r)$ of rank $r$ stable bundles on $C$ equipped with a flat connection. 
It is a smooth, quasi-projective scheme whose tangent space at a point $(E,\nabla)$ is the group $\bH^1_{dR}(C,\nabla)$.   There is a canonical isomorphism $\int_C\colon H^1(C,\omega_C)\to \bC$. The Atiyah-Bott form  is the composite of the wedge product
\begin{equation}\bH^1_{dR}(C,\nabla)\times \bH^1_{dR}(C,\nabla)\lRa{\wedge} \bH^2_{dR}(C,\nabla)\end{equation}
with the canonical maps
\begin{equation}\bH^2_{dR}(C,\nabla)\lRa{\isom} H^1(C,\lEnd_{\O_C}(E)\tensor \omega_C)\lRa{tr} H^1(C,\omega_C)\lRa{\int_C} \bC.\end{equation}

There is a forgetful map \begin{equation}
\label{forge}
\pi\colon\Flat_C(r)\to \Bun_C(r)\end{equation} to the moduli space of rank $r$, degree $0$ stable bundles on $C$. The tangent space at a point of $\Bun_C(r)$ is the cohomology group $\Ext^1_C(E,E)$.  The map $\pi$ is an affine bundle for the vector bundle over $\Bun_C(r)$ with fibres $\Hom_C(E,E\tensor \omega_C)$, which can be identified with the cotangent bundle of $\Bun_C(r)$ using the Serre duality pairing
\begin{equation}
\label{serre}\Hom_C(E,E\tensor\omega_C)\times \Ext^1_C(E,E)\lra \Ext^1(E,E\tensor \omega_C)\lRa{\tr} H^1(C,\omega_C)\lRa{\int_C} \bC.\end{equation}
 The short exact sequence \eqref{bass} defined by the derivative of the map $\pi$ can be identified with \eqref{snow}, and the following 
result is then immediate.

\begin{lemma}
The Atiyah-Bott form on $\Flat_C(r)$ is uniquely characterised by the following two properties:
\begin{itemize}
\item[(i)] The fibres of the map \eqref{forge} are Lagrangian,
\item[(ii)] At any point $(E,\nabla)$ the induced pairing between the vertical tangent space for the map \eqref{forge} and the  cotangent space $T^*_E \Bun_C(r)$ is the Serre duality pairing \eqref{serre}.
\end{itemize}
\end{lemma}

Note that in the case of bundles of rank $r=1$ the de Rham complex for $(E,\nabla)$ becomes the usual de Rham complex of $C$ and hence $\bH^1_{dR}(C,\nabla)=H^1(C,\bC)$. The Atiyah-Bott pairing can then be identified with the intersection form.

\subsection{Properties of  traces}

Let $X$ be a variety, and  $E$ a vector bundle on $X$. Taking the trace of endomorphisms defines a map $\tr\colon \lEnd_{\O_X}(E)\to \O_X$. Tensoring by a   line bundle $L$ and taking cohomology gives linear maps
 \begin{equation}\tr_E\colon \Ext_X^p(E,E\tensor L)\to H^p(X,L).\end{equation}
 In the proof of Theorem \ref{symplectic} we shall need the following properties of these maps, whose proofs we leave to the reader:
\begin{itemize}
\item[(T1)] Given bundles $E,F$ and a line bundle $L$, and elements $g\in \Ext_X^p(E,F)$ and $h\in \Ext_X^q(F,E\tensor L)$ there is an identity \begin{equation}\tr_E(h\circ g)=\tr_F((g\tensor L)\circ h)\in H^{p+q}(X,L).\end{equation}
\item[(T2)] Given  a bundle $E$ and line bundles $L,M$ and elements $f\in \Ext_X^p(E,E\tensor L)$ and $g\in \Ext_X^q(L,M)$ there is an identity
\begin{equation}g\circ \tr_E(f)=\tr_E((1_E\tensor g)\circ f).\end{equation}
\item[(T3)] Suppose given two bundles $E_1,E_2$ and a line bundle $L$, and an element $f\in \Ext^p_X(G,G\tensor L)$ where $G=E_1\oplus E_2$. Let $s_i\colon E_i\to G$  be the canonical inclusions, and $\pi_i\colon G\to E_i$   the canonical  inclusions. Then
\begin{equation} \tr_{G}(f)=\tr_{E_1}(\pi_1\circ f\circ s_1) + \tr_{E_2}(\pi_2\circ f\circ s_2).\end{equation}
\item[(T4)] Given a map of varieties $f\colon X\to Y$, a bundle $E$ on $Y$, a line bundle $L$ on $Y$ and an element $s\in \Ext^p(E,E\tensor L)$, there is an identity \begin{equation}\tr_{f^*(E)}(f^*(s))=f^* (\tr_E(s))\in H^p(X,f^*(L)).\end{equation}
\end{itemize}

\subsection{Preservation of symplectic forms}

Each of the maps $\pi_1,\cdots, \pi_5$ appearing in the diagram \eqref{biggy} and \eqref{seconddiag} carries a natural relative symplectic form, and we claim that the horizontal maps in these diagrams preserve these forms. For the most part these statements are rather obvious, but Theorem \ref{symplectic} below is quite non-trivial.

Let us consider each vertical map in turn. For the map $\pi_3$ we take the Atiyah-Bott form for rank 1 bundles with connection, restricted to the subset of anti-invariant connections. The tangent space to the fibres is $H^1(\Sigma,\bC)^-$ and the symplectic form is just the usual intersection form. We then take the same form on $\pi_4$ and by definition of the symplectic form $\omega$ on $M$ this induces the required relative symplectic form on the map  $\pi_5$.  

The map $\pi_1$ is equipped with a relative form whose restriction to each fibre is the Atiyah-Bott form discussed above. Note that since we are in the $\SL_2(\bC)$ setting we should impose trace free conditions appropriately. Since the right-hand square in \eqref{biggy} is Cartesian this induces a relative symplectic form on $\pi_2$.
In terms of the map $\gamma_\epsilon=\beta_\epsilon\circ\alpha^{-1}$ of Corollary \ref{cor} what is then left to prove is

\begin{thm}
\label{symplectic}
The pullback of the Atiyah-Bott form on $\cM_C(E,\nabla)$ by the rational map $\gamma_\epsilon$ is twice the Atiyah-Bott form on $\cM_\Sigma(L,\partial)$.
\end{thm}

\begin{proof}
We freely use notation from Section \ref{explicit}. 
At the level of bundles $\gamma_\epsilon$ is defined by $E=p_*(L)$. We have a map of short exact sequences in which all vertical arrows are isomorphisms
\begin{equation*}
\begin{gathered}
\xymatrix@C=1em{  
 \Hom_{\Sigma}(L,L\tensor \omega_\Sigma)_0 \ar[rr] && T_{(L,\partial)} \cM_\Sigma(L,\partial) \ar[rr] && \Ext_{\Sigma}^1(L,L)_0\ar[d]^{p_*}
 \\
 \Hom_{C}(E,E\tensor \omega_C)_0\ar[u]^{\gamma_{\epsilon,*}} \ar[rr] && T_{(E,\nabla)} \cM_C(E,\nabla)\ar[u]^{\gamma_{\epsilon,*}} \ar[rr] && \Ext_{C}^1(E,E)_0
 }\end{gathered}
 \end{equation*}

Take $L$ an anti-invariant line bundle on $\Sigma$ and set $E=p_*(L)$. Take elements
\begin{equation}v\in \Hom_C(E,E\tensor\omega_C), \qquad w\in \Hom^1_\Sigma(L,L).\end{equation}
We assume that $\tr(v)=0$ and $w$ is anti-invariant, meaning that $w\tensor 1_{\sigma^*(L)} + 1_L\tensor \sigma^*(w)=0$. What we must show is that
\begin{equation}\int_C \tr_E(v\circ p_*(w))= \int_{\Sigma} \tr_L(\alpha_*(v)\circ w).\end{equation}

By construction, $\alpha_*(v)$ is given by the composite 
\begin{equation}L\lRa{\chi_L} p^*(E)\tensor \omega_{\Sigma/C} \lRa{p^*(v)\tensor \omega_{\Sigma/C}} p^*(E)\tensor \omega_{\Sigma} \lRa{ \eta_L\tensor \omega_{\Sigma}} L\tensor \omega_{\Sigma}.\end{equation}  
Thus we have
\begin{equation}\int_{\Sigma} \tr_L(\alpha_*(v)\circ w)=\int_\Sigma \tr_L(\eta_L\tensor \omega_{\Sigma/C}\circ p^!(v)\circ \chi_L\circ w).\end{equation}

On the other hand, using the relation $\int_C \psi=\int_{\Sigma}(s\circ p^*(\psi))$ valid for all elements $\psi\in H^1(C,\omega_C)$ we can write
 \begin{equation}\int_{C} \tr_E(v\circ p_*(w))= \int_{\Sigma} s\circ p^*   \tr_{E} (v\circ p_*(w)) \stackrel{(T4)}{=}\int_{\Sigma} s\circ  \tr_{p^*(E)} (p^*(v)\circ p^* p_*(w))\end{equation}\begin{equation} \stackrel{(T2)}{=}\int_{\Sigma} \tr_{p^*(E)} \big((1_{p^*(E)}\tensor s)\circ p^*(v)\circ p^* p_*(w)\big)  =\int_{\Sigma} \tr_{p^*(E)} (g\circ f\circ p^*(v)\circ p^* p_*(w))\end{equation}\begin{equation} \stackrel{(T1)}{=}\int_{\Sigma} \tr_{L\oplus \sigma^*(L)} (f\tensor \omega_{\Sigma/C}\circ p^!(v)\circ p^! p_*(w)\circ g)\end{equation}
 \begin{equation}\label{early}\stackrel{(T3)}{=}\int_{\Sigma} \tr_{L} (\eta_L\tensor \omega_{\Sigma/C}\circ p^!(v)\circ \chi_L\circ w)+\int_{\Sigma} \tr_{\sigma^*(L)} (\eta_{\sigma^*(L)}\tensor \omega_{\Sigma/C}\circ p^!(v)\circ \chi_{\sigma^*(L)}\circ \sigma^*(w)),\end{equation}
where we used the naturality of $\chi$ as well as (T3) in the final step. The same argument together with the assumption $\tr_E(v)=0$ shows that
\begin{equation}\int_{\Sigma} \tr_L(\eta_L\tensor \omega_{\Sigma/C}\circ p^!(v)\circ \chi_L)+\int_{\Sigma} \tr_{\sigma^*(L)}(\eta_{\sigma^*(L)}\tensor \omega_{\Sigma/C}\circ p^!(v)\circ \chi_{\sigma^*(L)})=0.\end{equation}
For line bundles (T2) shows that the trace of the composite is the composite of the traces. Using the assumption that
$w$ is anti-invariant we see that the two terms in \eqref{early} are equal, which proves the claim.
\end{proof}


\section{Construction of the Joyce structure}
\label{seven}

In this section we finally construct the required meromorphic Joyce structure on the complex manifold $M=\cM(C,Q)$. The key ingredient is the isomonodromy connection on the map $\pi_1$, which is both flat and symplectic. We pull this back across the diagrams \eqref{biggy} and \eqref{seconddiag} using the  generic finiteness results of  Theorem \ref{birational}. This then gives the required non-linear connections $h_\epsilon$ on the tangent bundle of $M$.

\subsection{Non-linear connections}

Let $\pi\colon X\to Y$ be a smooth map of varieties. We can define the notion of a non-linear connection on $\pi$ exactly as before. Namely, the derivative of $\pi$ gives a short exact sequence
\begin{equation*}\label{bass2}0\lra T_{X/M}\lRa{i} T_X\lRa{\pi_*} \pi^*(T_Y)\lra 0,\end{equation*}
and we define a non-linear connection on $\pi$ to be a map of bundles $h\colon \pi^*(T_Y)\to T_X$ satisfying $\pi_*\circ h=\id$.
Let $D\subset X$ be an effective Cartier divisor with corresponding section $s_D\colon \O_X\to \O_X(D)$. We define a meromorphic connection on $\pi$ with poles along  $D$ to be a map of bundles $h\colon \pi^*(T_Y)\to T_X(D)$ satisfying $(\pi_*\tensor\O_X(D))\circ h=1_{\pi^*(T_Y)}\tensor s_D$.

 We shall need the following simple facts whose proofs we leave to the reader:
\begin{itemize}
\item[(C1)] Suppose given smooth maps $\pi\colon X\to Y$ and $\eta\colon Y\to Z$ and connections $h\colon \pi^*(T_Y)\to T_X$ on $\pi$ and $j\colon \eta^*(T_Z)\to T_Y$  on $\eta$. Then the composite  $h\circ \pi^*(j)\colon \pi^*\eta^*(T_Z)\to T_X$ is a connection on  the map $\eta\circ \pi\colon X\to Z$. 

\item[(C2)] If $\pi\colon X\to Y$ is {\'e}tale then $\pi_*^{-1}\colon \pi^*(T_Y)\to T_X$ is the unique connection on $\pi$.

\item[(C3)] Given a Cartesian square
\begin{equation*}
\begin{gathered}
\xymatrix@C=1em{  
 W\ar[d]_{\eta} \ar[rr]^{g}&& X\ar[d]^{\pi}   \\
 Z \ar[rr]^{f} && Y}
\end{gathered}\end{equation*}
with $\pi$ smooth, and a connection $h\colon \pi^*(T_Y)\to T_X$  on $\pi$, there is a unique connection $j\colon \eta^*(T_Z)\to T_W$ on $\eta$ such that \begin{equation}g_*\circ j =g^*(h)\circ \eta^*(f_*)\colon \eta^*(T_Z)\to g^*(T_X).\end{equation}


\item[(C4)] Suppose given a smooth map $\pi\colon X\to Y$ and an open subset $U\subset X$. Denote the inclusion map by $i\colon U\to X$, and set $\pi_0=\pi\circ i$. Then any connection $h_0\colon \pi_0^*(T_Y)\to T_U$ on $\pi_0$ extends to a meromorphic connection on $\pi$. More precisely, there is an effective divisor $D\subset X$ and a meromorphic connection $h\colon \pi^*(T_Y)\to T_X(D)$ on $\pi$ with poles along $D$ such that $i^*(h)=\id_{T_U}\tensor i^*(s_D)\circ h_0$.
\end{itemize}

\subsection{Isomonodromy connection}

Return to the diagram \eqref{biggy}, and recall that we are imposing the condition that the bundles $E$ are stable. In particular, the map 
\begin{equation}\pi_1\colon \cM(C,E,\nabla)\to \cM(C),\end{equation}
is smooth. There is  a flat non-linear connection on  this map known as the isomonodromy connection and constructed as follows. There is a map \begin{equation}\pi_0\colon \cM(C,\cE)\to \cM(C)\end{equation} whose fibre over  a curve $C$ is the space of $\SL_2(\bC)$ local systems on  $C$. The relative Riemann-Hilbert correspondence \cite[Theorem 2.23]{Del} gives an open embedding  $\mu\colon \cM(C,E,\nabla)\to \cM(C,\cE)$ sending a bundle with connection $(E,\nabla)$  to its monodromy local system.  The fact that the universal family of curves over $\cM(C)$ is locally trivial as a family of smooth surfaces defines a   Gauss-Manin connection on $\pi_0$, which by pullback along $\mu$ then induces the isomonodromy connection on $\pi_1$.

We shall need two properties of the isomonodromy connection. Firstly, despite the fact that the monodromy map $\mu$ is highly transcendental, the isomonodromy connection is nonetheless an algebraic object. The basic reason is that the condition for a relative connection $(\cE,\nabla)$ on a family of curves curve $f\colon \cC\to T$  to be isomondromic can be rephrased as the lifting of the relative flat connection on the bundle $\cE$ over $\cC$ to an actual flat connection, and the isomonodromy connection then arises from a zero curvature condition. A more abstract approach using crystals was explained by Simpson \cite[Section 8]{Simp2}.

The second property of the isomonodromy connection we need is that it is symplectic. Goldman \cite{Gold} proved that the Riemann-Hilbert map $\mu$ takes the Atiyah-Bott symplectic form on $\cM_C(E,\nabla)$ to a natural symplectic form on the character variety defined using group cohomology. The only important point for us is that this second symplectic form is defined topologically, and is therefore independent of the complex structure on $C$. It follows that the parallel transport maps for the isomonodromy connection preserve the Atiyah-Bott symplectic form on the fibres of $\pi_1$.

The right-hand square in \eqref{biggy}  is Cartesian so we can  pull back the isomonodromy connection using (C3) to obtain a  connection on $\pi_2$. Using (C1), (C2) and Theorem \ref{birational} we obtain a connection on an open subset of  $\pi_3$.
Note that the bundle of groups $J^2(C)$ over $\cM(C)$ acts by tensor product on the upper row of the diagram \eqref{biggy}. Tensoring $(E,\nabla)$ by an element $(P,\partial_P)\in J^2(C)$ multiplies the monodromy by  a homomorphism $H_1(C,\bZ)\to \{\pm 1\}$.  This action clearly preserves the isomonodromy connection,  and the connection on $\pi_3$ therefore descends along the  map $\tau$ appearing in \eqref{seconddiag}. Using (C4) we can extend it to a meromorphic connection. Continuing  across this diagram, we finally obtain a  meromorphic connection  $h_\epsilon$ on a dense open subset of   $\pi_5\colon T_M^{\hash}\to M$. 

\subsection{Joyce structure}
In the last section we showed how to construct, for each $\epsilon\in \bC^*$, a  non-linear connection $h_\epsilon$ on a dense open subset of $\pi_5$.  This connection is flat and symplectic  because it is a pullback of the isomonodromy connection. To produce a pre-Joyce structure it remains to prove that as $\epsilon\in \bC^*$ varies these connections form a $\nu$-pencil, i.e. that $h_\epsilon=h+\epsilon^{-1} v$ where $h=h_\infty$.

Consider the automorphism
\begin{equation*}r_\epsilon\colon \cM(C,E,\nabla,\Phi)\to \cM(C,E,\nabla,\Phi), \qquad (C,E,\nabla,\Phi)\mapsto (C,E,\nabla+\epsilon^{-1}\Phi,\Phi).\end{equation*}
Then  $\beta_\epsilon=\beta_\infty\circ r_{-\epsilon}$ and  it follows that $h_\epsilon=(r_\epsilon)_*\circ h$ as maps of bundles $\pi^*(T_M)\to T_X$. It will  be enough to show that $(r_\epsilon)_*=\id+\epsilon^{-1} (v\circ\pi_*)$. 

Recall the tautological differential $\lambda\in H^0(\Sigma,\omega_\Sigma)$ of Section \ref{quad} whose  periods \eqref{periodcoords} around a basis of cycles $(\gamma_1,\cdots,\gamma_n)\subset  H_1(\Sigma,\bZ)^-$ define local flat integral co-ordinates $z_i$ on $M$.  After  transferring along the birational map $\alpha$  the map $r_\epsilon$ becomes
\begin{equation*}r_\epsilon\colon \cM(C,Q,L,\partial)\to \cM(C,Q,L,\partial), \qquad (C,Q,L,\partial)\mapsto (C,Q,L,\partial+\epsilon^{-1}\lambda).\end{equation*}
Thus  $r_\epsilon$ is the operation of tensoring  $(L,\partial_L)\in J^\hash_{\br}(\Sigma)$ with   $(\O_{\Sigma},d+\epsilon^{-1}\lambda)\in J^\hash(\Sigma)^-$.    Note that the monodromy of this second connection  around a cycle $\gamma_i$ is just $\exp(\epsilon^{-1}z_i)$. We now transfer the automorphism $r_\epsilon$ across the diagram \eqref{seconddiag}.  The map $\sigma$ takes  a point of the space $\cM(C,Q,L,\partial)$ to the monodromy of the product \eqref{ukraine}. Thus  taking  fibre co-ordinates $\theta_i$ as in Section \ref{prejoyce}, we finally arrive at the automorphism   of $X=T_M$ given in local co-ordinates by $\theta_i\mapsto \theta_i+\epsilon^{-1} z_i$, and the claim  follows.


The final step is to check the compatibility between the period structure and the pre-Joyce structure.  Conditions (J1) and (J2)  of Definition \ref{joyce} hold by construction. It remains to consider (J3) and (J4).

For (J3) we consider the $\bC^*$ action on $M=\cM(C,Q)$ for which $t\in \bC^*$ acts by  $t\cdot (C,Q)=(C,t^2\cdot Q)$. Combining the induced action on $X=T_M$ with the rescaling action on the fibres as in the paragraph before Lemma \ref{jen} gives a $\bC^*$ action on $X=T_M$. It is not hard to see that this  descends to $X^{\hash}$, and that when transferred across the diagrams \eqref{biggy} and \eqref{seconddiag} it becomes the $\bC^*$ action on  $\cM(C,E,\nabla,\Phi)$ which sends $\Phi\mapsto t\Phi$ and leaves $(C,E,\nabla)$ fixed.\todo{Explain}  
We denote by $m_t\colon M\to M$ and $n_t\colon X\to X$ the resulting actions of $t\in \bC^*$. Note that the involution $\iota\colon X\to X$ coincides with $n_{-1}$.\todo{Explain}

After Lemma \ref{jen} to check (J3) it will be enough to show that for any vector field $v$ on $M$ we have $(n_t)_*(h(v))=h((m_t)_*(v))$. Taking $t=-1$ this will also imply  (J4). To prove this identity it is enough to show that $(n_t)_*(h(v))$ is a horizontal vector field for $h$. But this is clear by construction of $h$ since in the diagram \eqref{biggy} $\rho'\circ n_t=\rho'$.\todo{Explain that we can put an algebraic structure on $X^\hash$ so that everything is algebraic.}

\subsection{Restriction to the zero section}
\label{martha}

It was explained in \cite[Section  3.2]{Strachan} that the  involution property (iii) above implies that, when restricted to the zero-section $M\subset X=T_M$, the holomorphic Levi-Civita connection of the complex \hk structure on $X$ induces a flat, torsion-free connection $\nabla^J$ on the  tangent bundle of $M$. This connection was referred to in \cite[Section  7]{RHDT2} as the  linear Joyce connection, and is given in co-ordinates  by
\begin{equation}
\label{linearjoyce}\nabla^J_{\frac{\partial}{\partial z_i}} \Big(\frac{\partial}{\partial z_j}\Big)=\sum_{p,q} \eta_{qp}\cdot \frac{\partial^3 W}{\partial \theta_i \partial \theta_j \partial \theta_p}\Big|_{\theta=0}\cdot \frac{\partial}{\partial z_q}.
\end{equation}

Note however, that locating the poles of the Joyce structure constructed above  is a subtle problem, and in particular it is not clear whether the structure is regular along the zero-section $M\subset X=T_M$. This submanifold  $M\subset X$  is the fixed locus of the  involution $\iota$, and when transferred across the diagram \eqref{seconddiag} corresponds to  the multi-section of $\pi_3$ consisting of points satisfying $(L,\partial_L)^{\tensor 2}=(p^*(\omega_C),\partial_{\can})$. 
Understanding the properties of the Joyce structure at these points is difficult however, because they lie in the exceptional locus of the birational map $\alpha$.  Note that  in the one example that has been computed in detail \cite{A2}, the connection $\nabla^J$ is indeed well-defined, and turns out to be quite natural.  


\section{Good Lagrangian submanifolds}
\label{slices}

The spaces of complex and K{\"a}hler  parameters on a compact Calabi-Yau threefold are expected to  appear as complex Lagrangian submanifolds in the  stability space of the associated CY$_3$ triangulated category.  
It has been a long-standing question to try to abstractly characterise these submanifolds in stability space (see e.g. \cite[Section 7]{spaces}). 
  In this section we give a general definition of a  good Lagrangian submanifold $B\subset M$ in the base of a Joyce structure. We then prove that for the Joyce structures constructed in Section \ref{seven},  the submanifolds in $M=\cM(C,Q)$ obtained by fixing the curve $C$ and varying the quadratic differential $Q$ are  good Lagrangians in this sense. 

\subsection{General definition}

Consider a Joyce structure on a complex manifold $M$  and a complex Lagrangian submanifold $B\subset M$. Consider the normal bundle $\pi\colon N_B\to B$ fitting into the sequence
\begin{equation}
\label{normal2}0\lra T_B\lRa{i} T_M|_B\lRa{k} N_B\lra 0.\end{equation}
Recall the pencil of  connections $h_\epsilon=h+\epsilon^{-1}v$ on the bundle $\pi\colon X=T_M\to M$. For any complex submanifold $B\subset M$, and any $\epsilon^{-1}\in \bC$, the   connection $h_\epsilon$  restricts to a connection  $h_\epsilon|_B$ on the bundle $X_B=T_M|_B\to B$.



 \begin{definition}
 \label{good}
A complex Lagrangian  submanifold $B\subset M$ will be called good if  the restricted connection  $h_\epsilon|_B$  descends via the map $k\colon X_B\to N_B$ to a  connection $n$ on the normal bundle $\pi\colon N_B\to B$.
\end{definition}

To explain this condition in more detail, take  $x\in X_B$ with $\pi(x)=b\in B$.  The bundle map $k$ defines  a map of complex manifolds $k\colon X_B\to N_B$, and we set $y=k(x)\in N_B$.\begin{equation}
\label{snooze}
\begin{gathered}
\xymatrix@C=1em{  
 X\ar[d]_{\pi} && X_B\ar[d]\ar@{_{(}->}[ll]\ar[rr]^{k}&& N_B\ar[d]^{\pi}   \\
M  && B\ar@{_{(}->}[ll] \ar@{<->}[rr]^{=}  && B
}
\end{gathered}\end{equation}

Given a vector $w\in T_b B\subset T_b M$ the connection $h_\epsilon$ defines a lift $h_\epsilon (w)\in  T_x X_B\subset T_x X$, and we define $n(w)=k_*(h_\epsilon(w))\in T_y N_B$. Note that $n(w)$   is independent of $\epsilon$, since $k_*(v(w))=0$. The condition of Definition \ref{good} is that $n(w)$ depends only on $y\in N_B$, not on the element $x\in X_B$ satisfying $k(x)=y$. When this condition holds the map $n$ defines a connection on $\pi\colon N_B\to B$. 

\begin{lemma}
\label{goody}
If $B\subset M$ is a good Lagrangian then the induced connection $n$ on  the normal bundle $\pi\colon N_B\to B$ is flat. \end{lemma}

\begin{proof}
 Recall that if $f\colon M\to N$ is a map of complex manifolds, and $u,w$ are vector fields on $M,N$ respectively, then $u,w$ are said to be $f$-related if $f_*(u_m)=w_{f(m)}$ for all $m\in M$. Given vector fields $u_1,u_2$ on $M$ which are $f$-related to vector fields $w_1,w_2$ on $N$ it is easily checked that $[u_1,u_2]$ is $f$-related to $[w_1,w_2]$. We will apply this to the map $k\colon X_B\to N_B$.

Given a vector field $u$ on $B$, we can extend it to a vector field on $M$ which we also denote by $u$. We can then use the connection $h_\epsilon$ to lift it to the vector field $h_\epsilon(u)$ on $X$. The restriction of this vector field to $X_B$ is a vector field on $X_B$, and is independent of the chosen extension. The good Lagrangian condition states that this vector field on $X_B$ is $k$-related to a vector field on $N_B$, which by definition is $n(u)$.

The connection $h_\epsilon$ being flat is the condition that for any vector fields $u_1,u_2$ on $M$ we have  $h_\epsilon([u_1,u_2])=[h_\epsilon(u_1),h_\epsilon(u_2)]$.  But then it follows that $h_\epsilon([u_1,u_2])$ is $k$-related to $[n(u_1),n(u_2)]$, which by definition of $n$ implies that $n([u_1,u_2])=[n(u_1),n(u_2)]$ and hence that the connection $n$ is flat. 
\end{proof}

To express the good Lagrangian condition  more concretely,  take local Darboux co-ordinates $(z_1,\cdots,z_{2d})$ on $M$   and assume that $B\subset M$ is given by the equations $z_{d+1}=\cdots=z_{2d}=0$, and that $\omega_{pq}=\pm 1$ if $q-p=\pm d$ and is otherwise zero. 
Lifting the vector fields $\partial/\partial z_i$  with $1\leq i \leq d$ from $M$ to $X$ as in \eqref{above} gives
\begin{equation}
\label{vecf}v_i=\frac{\partial}{\partial \theta_i}, \qquad h_i=\frac{\partial}{\partial z_{i}} + \sum_{j=1}^d \Big(\frac{\partial^2 W}{\partial \theta_{i} \partial \theta_{j+d}} \cdot \frac{\partial}{\partial \theta_{j}}-\frac{\partial^2 W}{\partial \theta_{i} \partial \theta_j} \cdot \frac{\partial}{\partial \theta_{j+d}}\Big).\end{equation} 
Applying the projection $k\colon T_M|_B \to N_B$ amounts to setting $\partial/\partial \theta_i=0$ for $1\leq i\leq d$.  The condition of Definition \ref{good} is then that for $1\leq i\leq d$ the result of this projection should be independent of the co-ordinates $\theta_i$ for $1\leq i \leq d$. This is equivalent to
\begin{equation}\label{vanish}\frac{\partial^3 W}{\partial \theta_{i} \partial \theta_{j} \partial \theta_{k}}=0 \qquad 1\leq i,j,k\leq d\end{equation}
along the locus $z_{d+1}=\cdots=z_{2d}=0$.

There is a canonical real structure on the  tangent bundle $T_M$ whose fixed locus $T_M^{\bR}\subset T_M$ is the real span of the  integral affine structure $T_M^\bZ\subset T_M$. We call a complex Lagrangian $B\subset M$ non-degenerate if $T_B\cap T_M^{\bR}|_B =(0)\subset T_M|_B$.
When this holds, the restriction of the map $k$  to the lattice $T_M^{\bZ}|_B\subset T_M|_B$ is injective, and we denote its  image by  $N_B^{\bZ}\subset  N_B$. When the Lagrangian $B$ is both good and non-degenerate the connection  $n$ of Lemma \ref{goody} descends to a connection on the projection $\pi\colon N_B/N_B^{\bZ}\to B$  whose fibres are compact tori $\bC^d/\bZ^{2d}\isom (S^1)^{2d}$.
%


\subsection{Class $S[A_1]$ examples}

Consider the Joyce structure on the space $M=\cM(C,Q)$ constructed in this paper. Let us fix a curve $C\in \cM(C)$ and consider the Lagrangian submanifold $B=\cM_C(Q)\subset M$ which is the  corresponding fibre  of the projection $\rho\colon \cM(C,Q)\to \cM(C)$. Thus $B\subset H^0(C,\omega_C^{\tensor 2})$ parameterises quadratic differentials on $C$ with simple zeroes.

Note that if $Q'\in H^0(C,\omega_C^{\tensor 2})$ is a  quadratic differential on $C$, then $p^*(Q')$ vanishes to order two along the branch divisor $R\subset \Sigma$. It follows that the tangent space to $B$ at a point $(C,Q)$ can be identified with $H^0(\Sigma,\omega_\Sigma)^-$ via the  map
\begin{equation}T_b  B=H^0(C,\omega_C^{\tensor 2})\to H^0(\Sigma,\omega_{\Sigma})^-, \qquad  Q'\mapsto {p^*(Q')}/{2\lambda},\end{equation}
where $\lambda$ is the tautological 1-form on $\Sigma$ appearing in \eqref{delta}.
Under the isomorphism of Theorem \ref{existperiods} the sequence \eqref{normal2} then corresponds to the Hodge filtration
\begin{equation} \label{goodie}0\lra H^0(\Sigma,\omega_\Sigma)^-\lRa{i} H^1(\Sigma,\bC)^-\lRa{k} H^1(\Sigma,\O_{\Sigma})^-\lra 0.\end{equation}
It follows from the isomorphism \eqref{mapp} that the fibers  \begin{equation}H^1(\Sigma,\O_{\Sigma})^-/\tilde{H}^1(\Sigma,\bZ)= P^\sharp(\Sigma)/H^0(\Sigma,\omega_\Sigma)^-\end{equation}
of the map  $\pi\colon N_B/N_B^{\bZ}\to B$  are the  Prym varieties $P(\Sigma)$ appearing in Section \ref{prym}. 

\begin{lemma}
\label{donee}
For each curve $C$ the submanifold $B=\cM_C(Q)\subset M=\cM(C,Q)$ is a good Lagrangian. The horizontal leaves of the induced meromorphic flat connection on $\pi\colon N_B/N_B^{\bZ}\to B$ are defined by the condition that $E=p_*(L)$ is constant.
\end{lemma}

\begin{proof}
We use the notation $\cM_C(Q,L,\partial)$ to denote the space parameterising data $(Q,L,\partial)$ on the fixed curve $C$, and similarly for $\cM_C(Q,E,\nabla_\epsilon)$, etc. 
Let $w$ be a vector field on $B\subset M$ and let $u=h_\epsilon(w)$ be the lift to a vector field on $X_B\subset X$. Transferring across the diagram \eqref{seconddiag} we can consider $u$ to be a vector field on $\cM_C(Q,L,\partial)$, and we must show that it descends to the space $\cM_C(Q,L)$. That is, the flow of the line bundle $L$ under $u$ should be independent of the connection $\partial$. Passing through the diagram \eqref{biggy} we can view $u$ as a vector field on $\cM_C(Q,E,\nabla_\epsilon)$, and we must show that it descends to the space $\cM_C(Q,E)$.

By definition, the connection $h_\epsilon$ on the projection $\pi_2$ is  pulled back from the isomonodromy connection on $\pi_1$. Since $\rho_*(w)=0$, it follows that  $\rho'_*(u)=0$. That is, $u$ is obtained by keeping the pair $(E,\nabla_\epsilon)$ on $C$ fixed as $Q$ varies with $w$. It is then clear that  $u$ descends to $\cM_C(Q,E)$, and the result follows.\end{proof}

Consider the diagram of moduli spaces
\begin{equation*}
\label{sunny}
\begin{gathered}
\xymatrix@C=1em{  
 N_B/N_B^{\bZ}\ar[d]_{\pi}  && \cM_C(Q,L)\ar[ll]_{\tau}\ar[d] \ar[rr]^{\kappa}&& \cM_C(E,\Phi)\ar[d]  \\
 B \ar@{<->}[rr]^{=} &&\cM_C(Q) \ar@{<->}[rr]^{=}&& \cM_C(Q)
}
\end{gathered}\end{equation*}
Here $\kappa$ is the isomorphism defined by the usual spectral construction sending a line bundle $L$ on $\Sigma$ to the Higgs bundle $(E,\Phi)$ on $C$, and $\tau$ is induced by the corresponding map from \eqref{seconddiag}. The forgetful map  $\cM_C(E,\Phi)_0\to \cM_C(E)$ can be identified with the cotangent bundle of $\cM_C(E)$, and according to Lemma \ref{donee}, when transferred across the diagram \eqref{sunny},  the fibres of this map become the horizontal leaves of the connection $n$. 

 \begin{appendix}
\section{Definition of the moduli spaces}
\label{moduli}

In this section we give detailed constructions of  the moduli spaces appearing in the diagram \eqref{biggy}. All schemes are over $\Spec(\bC)$.  We fix a genus $g>1$ and a level $\ell>2$ throughout. We use the terminology bundle  for locally-free sheaf of finite rank, and line bundle for invertible sheaf.

\subsection{Curves with level structure}

A family of genus $g$ curves is a smooth proper map of schemes $f\colon \cC\to S$ of relative dimension 1 whose geometric fibres are connected and of genus $g$. Given a map of schemes $s\colon S'\to S$ we can pull back the family by forming the Cartesian diagram
\begin{equation}
\label{stack}
\xymatrix@C=1em{  \cC'\ar[d]_{f'} \ar[rr]^{t} &&\cC\ar[d]^{f} \\ S'\ar[rr]^{s} && S}\end{equation}

Given a family of genus $g$ curves $f\colon \cC\to S$ as above, there is a locally-constant sheaf of free $\bZ/\ell$-modules $\cV_f=\bR^1 f_*(\bZ/\ell)$ on $S$. This construction commutes with base-change: given a diagram \eqref{stack} there is a canonical isomorphism $\cV_{f'}\isom s^*(\cV_f)$. In particular, the  pullback of $\cV_f$ to a $\bC$-valued point of $S$ is the cohomology group $H^1(C,\bZ/\ell)\isom (\bZ/\ell)^{\oplus 2g}$ of the corresponding genus $g$ curve $C$.  The intersection form on these cohomology groups defines a  skew-symmetric $\bZ/\ell$-bilinear  form \begin{equation}\varpi_f\colon \cV_f\times\cV_f\to \bZ/\ell.\end{equation}

Let $V$ be the free $\bZ/\ell$-module on the symbols $a_1,b_1, \cdots,a_g,b_g$ equipped with the standard skew-symmetric form defined by \begin{equation}\varpi(a_i,a_j)=0=\varpi(b_i,b_j), \qquad \varpi(a_i,b_j)=\delta_{ij}, \qquad 1\leq i,j\leq g.\end{equation}
 A level $\ell$ structure on the family $f$ is then defined to be an isomorphism of $\bZ/\ell$-modules $\theta\colon V\to H^0(S,\cV_f)$ relating the forms $\varpi$ and $\varpi_f$. Given a diagram \eqref{stack}, a level $\ell$ structure on the family $f$ defines a pulled back level structure on the family $f'$ in the obvious way. 

Two families of curves $f_i\colon \cC_i\to S$ with level $\ell$ structure $\theta_i$ are isomorphic if there is an isomorphism $g\colon \cC_1\to \cC_2$ satisfying $f_2\circ g = f_1$ and preserving the level structures in the obvious way.  There is a functor ${\it M}(g,\ell)\colon ({\rm Sch}/\bC)^{\rm op}\to \Sets$ by sending a scheme $S$ to the set of isomorphism classes of families of genus $g$ curves over $S$ equipped with level $\ell$ structure. 

\begin{thm}
When $g>1$ and $\ell>2$  the functor $\it {M}(g,\ell)$ is represented by a smooth quasi-projective scheme $\cM(g,\ell)$.
\end{thm}

\begin{proof}
This appears to be standard, although it is hard to find a complete proof in the literature. Grothendieck \cite[\S 2]{Groth} shows that $\it{M}(g,\ell)$ is representable by an algebraic space for $\ell\gg 0$, and attributes to Serre the statement that $\ell>2$ is sufficient. Mumford {\it et al}  \cite[Theorem 7.9]{GIT} prove the analogous result on  moduli spaces of abelian varieties, which  implies that  $\it {M}(g,\ell)$ is representable by a quasi-projective scheme for $\ell\gg 0$, and again attribute to Serre the statement that  $\ell>2$ is sufficient.
\end{proof} 

The closed points of $\cM(g,\ell)$ parameterise smooth projective genus $g$ curves $C$  equipped with a choice of symplectic basis in the homology group $H_1(C,\bZ/\ell)$.

\subsection{Quadratic differentials}

Given a family of genus $g$ curves $f\colon \cC\to S$ we denote by $\omega_{\cC/S}$ the relative cotangent bundle. Given a Cartesian diagram \eqref{stack} there is a canonical isomorphism $\omega_{\cC'/S'}\isom g^*(\omega_{\cC/S})$.

If $C$ is any fibre of $f$ then Serre duality gives $H^1(C,\omega_C^{\tensor 2})=H^0(C,T_C)^*=0$. Using cohomology and base-change it follows that $f_*(\omega_{\cC/S}^{\tensor 2})$ is a vector bundle on $S$. By a quadratic differential on the family of curves $f$ we  mean a section of this vector bundle. Note that
\begin{equation}H^0(S, f_*(\omega_{\cC/S}^{\tensor 2}))=H^0(\cC,\omega_{\cC/S}^{\tensor 2}).\end{equation}
Applying this construction to the universal family of curves defines a vector bundle $\cE$  on the space $\cM(g,\ell)$. 
 
Define a functor ${\it Quad}(g,\ell)\colon \Sch/\bC^{op}\to \Sets$ by sending a scheme $S$ to the set of isomorphism classes of families of genus $g$  curves equipped with level structures and quadratic differentials. 

\begin{lemma}
The functor ${\it Quad}(g,\ell)$ is represented by smooth quasi-projective variety $\Quad(g,\ell)$ which is the total space of the vector bundle $\cE$ over $\cM(g,\ell)$.
\end{lemma}

The closed points of $\Quad(g,\ell)$ parameterise smooth projective genus $g$ curves $C$ equipped with a choice of symplectic basis in the homology group $H_1(C,\bZ/\ell)$ and a quadratic differential $Q\in H^0(C,\omega_C^{\tensor 2})$.

 Consider a family of genus $g$ curves $f\colon \cC\to S$ and a section $Q\in H^0(\cS,\omega_{\cC/S}^{\tensor 2})$.  The relative critical locus of $Q$ is a closed subscheme of $\cC$, and since $f$ is proper, its image  is a closed subscheme in $S$.  This construction commutes with base-change. Applying it to the universal family defines a closed subscheme of $\Quad(g,\ell)$. We define $\Quad_0(g,\ell)\subset \Quad(g,\ell)$ to be the complementary open subscheme. By definition, a closed point of  $\Quad(g,\ell)$ lies in this open subscheme precisely if the corresponding quadratic differential $Q\in H^0(C,\omega_C^{\tensor 2})$ has simple zeroes.
 
\subsection{Bundles, Higgs fields and flat connections}
\label{higgs}

Let $f\colon \cC\to S$ be a family of genus $g$ curves over a scheme $S$.  We now give the definitions of the relative moduli spaces of bundles, Higgs bundles and flat connections we will need. 

 A family of rank $r$ bundles on $f$ is simply a rank $r$ bundle $E$ on $\cC$. Two such families $E_i$ are equivalent if there is a line bundle $L$ on $S$ and an isomorphism $\theta\colon E_1\to E_2 \tensor f^*(L)$. A family  $E$  is said to have trivial determinant if the associated family of rank 1 bundles $\wedge^r(E)$ is equivalent to the trivial family $\O_\cC$. 

A family of rank $r$ Higgs bundles is a bundle $E$ on $\cC$ equipped with a relative Higgs field $\Phi\colon E\to E\tensor \omega_{\cC/S}$. Two such families $(E_i,\Phi_i)$ are equivalent if there is an isomorphism $\theta\colon E_1\to E_2 \tensor f^*(L)$ which intertwines $\Phi_1$ and $\Phi_2\tensor 1_{f^*(L)}$. A family $(E,\Phi)$  has trivial determinant if the associated family of rank 1 Higgs bundles  $(\wedge^r(E),\wedge^r(\Phi))$ is equivalent to the trivial family $(\O_{\cC},0)$. (This is the usual condition that the Higgs bundle has zero trace).

A family of rank $r$ flat connections on $f$ is a bundle $E$ on $\cC$ equipped with a relative connection $\nabla\colon E\to E\tensor \omega_{\cC/S}$. Two such families $(E_i,\nabla_i)$ are equivalent if there is an isomorphism $\theta\colon E_1\to E_2 \tensor f^*(L)$ which is flat for the induced relative connection on $\mathcal{H}om_{\O_{\cC}}(E_1,E_2\tensor f^*(L))$. A family $(E,\nabla)$ has trivial determinant if the corresponding family of rank 1 flat connections  $(\wedge^r(E),\wedge^r(\nabla))$ is equivalent to the trivial family $(\O_\cC,d)$.

In all cases we say that a family on $\cC$ as stable if the restriction of the bundle $E$ to each geometric fibre of $f$ is stable. Note that for Higgs bundles and flat connections this is strictly stronger than the usual notion of stability. We use the stronger notion to ensure that the forgetful maps appearing in the diagram \eqref{biggy} are well-defined. Since stability is an open condition it corresponds to passing to open subsets of the usual moduli spaces.

Fix again a family of genus $g$ curves $f\colon \cC\to S$ and assume $S$ to be of finite type over $\bC$. 
There is a functor \begin{equation}\it{Bun}(\cC/S,r)\colon (\rm{Sch}/S)^{\rm op}\to \Sets\end{equation} which sends a map $m\colon T\to S$ to the set of equivalence classes of stable families of rank $r$  bundles with trivial determinant on the pulled back family $f_T\colon\cC\times_S T\to T$.  We define moduli functors ${\it Higgs}(\cC/S,r)$ and $\it{Flat}(\cC/S,r)$ in the same way.

\begin{thm}
The  functors ${\it Bun}(\cC/S,r)$, $\it{Higgs}(\cC/S,r)$ and $\it{Flat}(\cC/S,r)$ are co-representable by schemes ${\Bun}(\cC/S,r)$, ${\Higgs}(\cC/S,r)$ and ${\Flat}(\cC/S,r)$ respectively. Each of these schemes is  smooth and quasi-projective over $S$. The obvious forgetful maps ${\Higgs}(\cC/S,r)\to {\Bun}(\cC/S,r)$ and ${\Flat}(\cC/S,r)\to {\Bun}(\cC/S,r)$ are smooth.
\end{thm}

\begin{proof}
The co-representability follows from the results of Simpson \cite{Simp1}. The other statements are easy and well-known.
\end{proof}

We apply these results to the universal family of curves over the moduli space $\cM(g,\ell)$. We denote the resulting moduli spaces as $\Bun(g,\ell,r)$, $\Higgs(g,\ell,r)$ and $\Flat(g,\ell,r)$. They are smooth quasi-projective schemes. Note for example that $\Bun(g,\ell,r)$ co-represents the functor which sends a scheme $S$ to the set of equivalence classes of pairs $(f,E)$ consisting of a family of genus $g$ curves $f\colon \cC\to S$ and a stable family of rank $r$ bundles with trivial determinant  $E$ over $f$. Similar remarks apply to $\Higgs(g,\ell,r)$ and $\Flat(g,\ell,r)$. 

Given a family of Higgs fields $\Phi\colon E\to E\tensor \omega_{\cC/S}$ on a family of genus $g$ curves $f\colon \cC\to S$ we can define a quadratic differential $Q\in H^0(\cC, \omega_{\cC/S}^{\tensor 2})$ by setting  $Q=\tfrac{1}{2}\tr(\Phi^2)$. This defines a map   \begin{equation}
\label{hugs}\Higgs(g,\ell,r)\to \Quad(g,\ell)\end{equation} and we define $\Higgs_0(g,\ell,r)$ to be the inverse image of the open subset $\Quad_0(g,\ell)$ of quadratic differentials with simple zeroes. 

\subsection{Anti-invariant branched connections}

Consider a family  of genus $g$ curves $f\colon \cC\to S$ equipped with a quadratic differential $Q\in H^0(S,\omega_{\cC/S}^{\tensor 2})$  with no relative critical points. We can form a double cover $p\colon \Sigma\to \cC$ by writing the equation $y^2=Q$ inside the total space of the bundle $\omega_{\cC/S}$. This construction commutes with base-change in the obvious way. There is a covering involution $\sigma\colon \Sigma\to \Sigma$  and a branch divisor $R\subset \Sigma$ which is flat over $S$. 
The composite $g=f\circ p\colon \Sigma\to S$  is a family of smooth genus $4g-3$ curves. 

A family of branched connections on $g\colon \Sigma\to S$ is a line bundle $L$ on $\Sigma$ equipped with a relative meromorphic connection $\partial \colon L\to L\tensor \omega_{\Sigma/S} (R)$. Two such families $(L_i,\partial_i)$ are equivalent if there is a line bundle $N$ on $S$ and an isomorphism $\theta\colon L_1\to L_2 \tensor g^*(N)$ which is flat for the induced relative meromorphic connection on $\mathcal{H}om_{\O_{\Sigma}}(L_1,L_2\tensor g^*(N))$.
A family of branched connections $(L,\partial)$ is anti-invariant if the branched connection $(L,\partial)\tensor \sigma^*(L,\partial)$ is equivalent to the family of branched connections $(p^*(\omega_C),\partial_{can})$. 

There is a functor \begin{equation}\it{Flat^{\br}}(\Sigma/S)\colon (\rm{Sch}/S)^{\rm op}\to \Sets\end{equation} which sends a map $m\colon T\to S$ to the set of equivalence classes of anti-invariant families of branched connections on  the pulled back family $g_T\colon\Sigma\times_S T\to T$.  

\begin{thm}
The  functor $\it{Flat^{\br}}(\Sigma/S)$ is  representable by a scheme  ${\Flat^{\br}}(\Sigma/S)$ which is smooth and quasi-projective over $S$.
\end{thm}

\begin{proof}
For moduli spaces of flat connections with logarithmic singularities we can refer to Nitsure \cite{Nitsure}, but since we are dealing with rank 1 connections this is really over-kill. For a more elementary approach we can pass to an {\'e}tale cover of the functor $\it{Flat^{\br}(\Sigma/S)}$ by adding the data of a square-root of the line bundle $\omega_{\cC/S}$. Then, as in the proof of Lemma \ref{lemmaabove}, we can replace the branched connections $(L,\partial_L)$ with regular connections $(M,\partial_M)$. 
\end{proof}

We can apply the above construction to the universal family  over  $S=\Quad_0(g,\ell)$. We denote the resulting moduli space by $\Flat^{\br}(g,\ell)$. It is a smooth quasi-projective variety. It represents the functor which sends a scheme $S$ to the set of equivalence classes of quadruples $(f,Q,L,\partial)$ consisting of a family of genus $g$ curves $f\colon \cC\to S$ equipped with a quadratic differential $Q\in H^0(\cC,\omega_{\cC/S}^{\tensor 2})$ with simple zeroes, and a family of anti-invariant branched connections $(L,\partial)$ on the associated family of spectral curves $g\colon \Sigma\to S$.

\subsection{Moduli spaces and maps}

We can now define the moduli spaces and maps appearing in the diagram \eqref{biggy}. Firstly we set $\cM(C)=\cM(g,\ell)$ and $\cM(C,Q)=\Quad_0(g,\ell)$. Then we
take rank $r=2$ and set $\cM(C,E)=\Bun(g,\ell,2)$ and \begin{equation}\cM(C,E,\nabla)=\Flat(g,\ell,2), \qquad \cM(C,E,\Phi)=\Higgs_0(g,\ell,2).\end{equation}
We define $\cM(C,E,\nabla,\Phi)$ to be the fibre product of the obvious forgetful maps
\begin{equation*}
\label{stack2}
\xymatrix@C=1em{  \cM(C,E,\nabla,\Phi)\ar[d] \ar[rr]&&\cM(C,E,\Phi)\ar[d] \\ \cM(C,E,\nabla) \ar[rr] &&\cM(C,E). }\end{equation*}
Similarly we define $\cM(C,Q,E,\nabla)$ as the fibre product
\begin{equation*}
\label{stack3}
\xymatrix@C=1em{  \cM(C,Q,E,\nabla)\ar[d] \ar[rr] &&\cM(C,Q)\ar[d] \\ \cM(C,E,\nabla) \ar[rr] &&\cM(C) }\end{equation*}
Finally we take $\cM(C,Q,L,\partial)$ to be the open subvariety of $\Flat^{\br}(g,\ell)$ defined by the condition that $E=p_*(L)$ is stable.\todo{Improve this} Each of these spaces are smooth quasi-projective varieties.

The maps in the diagram \eqref{biggy} are just the obvious forgetful maps, with the exception of $\alpha$ and $\beta_\epsilon$. 
To define $\alpha$ we follow the same procedure in the text in the relative setting.
To define $\beta_\epsilon$ note  that $\beta_\epsilon=\beta_\infty\circ r_\epsilon$ where $r_\epsilon$ is the automorphism of the space $\cM(C,E,\nabla,\Phi)$ defined by $\nabla\mapsto \nabla+\epsilon^{-1}\Phi$. The map $\beta_\infty$ is given by the rule  $(C,E,\nabla,\Phi)\mapsto (C,Q,E,\nabla)$ and is induced using the above fibre-product diagrams from  the map \eqref{hugs}.

\end{appendix}



\begin{thebibliography}{101}









\bibitem{All0} D. Allegretti, Voros symbols as cluster co-ordinates,  J. Topol. 12 (2019), no. 4, 1031--1068.


\bibitem{All2} D. Allegretti, Stability conditions, cluster varieties, and Riemann-Hilbert problems from surfaces,  Adv. Math. 380 (2021), paper no. 107610, 62 pp


\bibitem{A} D. Arinkin, On $\lambda$-connections on a curve where $\lambda$ is a formal parameter, Math. Res. Lett. 12 (2005), no. 4, 551--565.






\bibitem{BNR} A. Beauville, M.S. Narasimhan and S. Ramanan, Spectral curves and the generalised theta divisor, J. Reine Angew. Math. 398 (1989), 169--179. 









\bibitem{spaces} T. Bridgeland, Spaces of stability conditions, Algebraic geometry (Seattle, 2005). Part 1, 1--21, Proc. Sympos. Pure Math., 80, Amer. Math. Soc.  (2009). 
\bibitem{RHDT1} T. Bridgeland, Riemann-Hilbert problems  from Donaldson-Thomas invariants, Invent. Math. 216 (2019), no. 1, 69--124.


\bibitem{RHDT2} T. Bridgeland, Geometry from Donaldson-Thomas invariants, Proc. Sympos. Pure Math., 103.2
AMS, 2021, 1--66.

\bibitem{A2} T. Bridgeland and D. Masoero, On the monodromy of the deformed cubic oscillator, Math. Ann. 385 (2023), no. 1-2, 193--258.

\bibitem{BS} T. Bridgeland and I. Smith, Quadratic differentials as stability conditions, Publ. Math. Inst. Hautes {\'E}tudes Sci. 121 (2015), 155--278. 

\bibitem{Strachan} T. Bridgeland and I.A.B. Strachan, Complex hyperk{\"a}hler structures defined by Donaldson-Thomas invariants, Lett. Math. Phys. 111 (2021), no. 2, paper no. 54, 24 pp.








\bibitem{Del} P. Deligne, {\'E}quations diff{\'e}rentielles {\`a} points singuliers r{\'e}guliers, 
Lecture Notes in Mathematics, Vol. 163. Springer-Verlag,  1970. iii+133 pp.



\bibitem{DP} R. Donagi and T. Pantev, Geometric Langlands  and non-abelian Hodge theory, Surveys in differential geometry. Vol. XIII. Geometry, analysis, and algebraic geometry: forty years of the Journal of Differential Geometry, 85--116,
Surv. Differ. Geom., 13, 2009. 




\bibitem{DM} M. Dunajski and L. Mason, Hyperk{\"a}hler hierarchies and their twistor theory, Commun. Math. Phys. 213 (2000) 641--672


\bibitem{FG} V. V. Fock and A. B. Goncharov,  Moduli spaces of local systems and higher Teichm\"uller theory, Publ. Math. Inst. Hautes {\'E}tudes Sci., {103}(1) (2006), 1--211.

\bibitem{FT} V.V. Fock and A. Thomas, Higher complex structures, Int. Math. Res. Not. (2021), no. 20, 15873--15893.



\bibitem{G} D. Gaiotto, Opers and TBA, arxiv.

\bibitem{GMN1} D. Gaiotto, G. Moore and A. Neitzke, Four-dimensional wall-crossing via three-dimensional field theory. Comm. Math. Phys. 299 (2010), no. 1, 163--224.

\bibitem{GMN2} D. Gaiotto, G. Moore and A. Neitzke, Wall-crossing, Hitchin systems, and the WKB approximation. Adv. Math. 234 (2013), 239--403.

\bibitem{Gold} W. M. Goldman, The symplectic nature of fundamental groups of surfaces, Adv. in Math. 54 (1984), no. 2, 200--225.




\bibitem{Gotay} M. J. Gotay, R. Lashof, J. {\'S}niatycki, and A. Weinstein, Closed forms on symplectic fibre bundles, Comment. Math. Helv. 58 (1983), no. 4, 617--621.

 \bibitem{Groth} A. Grothendieck, Techniques de construction en g{\'e}om{\'e}trie analytique. I. Description axiomatique de l'espace de Teichm{\"u}ller et de ses variantes. S{\'e}minaire Henri Cartan, Volume 13 (1960-1961) no. 1, Talk no. 7 et 8, 33 p.
 
\bibitem{Haiden} F. Haiden, 3d Calabi--Yau categories for Teichm{\"u}ller theory, arxiv.


\bibitem{Hitchin} N. Hitchin, The self-duality equations on a Riemann surface, Proc. London Math. Soc. (3) 55 (1987), no. 1, 59--126.

\bibitem{Hitchin2} N. Hitchin, Stable bundles and integrable systems.
Duke Math. J.54(1987), no.1, 91–114.

\bibitem{HN} L. Hollands and A. Neitzke, L. Hollands and A. Neitzke, Spectral Networks and Fenchel-Nielsen Coordinates, Lett.
Math. Phys. 106 (2016) 811--877.


%









\bibitem{HolGen} D. Joyce, Holomorphic generating functions for invariants counting coherent sheaves on Calabi-Yau 3-folds, Geom. Topol. 11 (2007), 667--725. 

\bibitem{KS} M. Kontsevich and Y. Soibelman, Affine structures and non-archimedean analytic spaces, Progr. Math., 244
(2006), 321--385.




\bibitem{DL} G. Laumon, Un analogue global du c{\^o}ne nilpotent, Duke Math. J. 57 (1988) 667--671.




 
 \bibitem{mum} D. Mumford, Prym varieties I, Contributions to analysis (a collection of papers dedicated to Lipman Bers), pp. 325--350. Academic Press, 1974. 



\bibitem{GIT} D. Mumford, J. Fogarty and F. Kirwan, 
Geometric invariant theory, Third edition, 
Ergeb. Math. Grenzgeb. (2), 34,
Springer-Verlag, Berlin, 1994. xiv+292 pp.


\bibitem{Nik1} N. Nikolaev, Abelianisation of logarithmic $\mathfrak{sl}_2$--connections, Selecta Math. (N.S.) 27 (2021), no. 5, no. 78, 35 pp.


\bibitem{Nitsure} N. Nitsure, Moduli of Semistable Logarithmic Connections, J. of the Amer. Math. Soc., Vol. 6, No. 3, (1993), pp. 597--609.





\bibitem{P} J. F. Pleba{\'n}ski, Some solutions of complex Einstein equations, J. Math. Phys. 16 (1975)
2395--2402.


\bibitem{Simp1} C.  Simpson,  Moduli of representations of the fundamental group of a smooth projective variety I,
Inst. Hautes {\'E}tudes Sci. Publ. Math. (1994), no. 79, 47--129.

\bibitem{Simp2} C. Simpson,  Moduli of representations of the fundamental group of a smooth projective variety II, 
Inst. Hautes {\'E}tudes Sci. Publ. Math. (1994), no. 80, 5--79.



\bibitem{Veech} W. Veech, Moduli spaces of quadratic differentials, J. d'Analyse Math.  55, pp. 117--171 (1990).




\bibitem{Z}  M. Zikidis, Joyce structures on spaces of meromorphic quadratic differentials, in preparation.

 
\end{thebibliography}
\end{document}